\documentclass[11pt]{article}
\usepackage{amsmath,amssymb,amsthm}

\newtheorem{thm}{Theorem}[section]
\newtheorem{lem}[thm]{Lemma}
\newtheorem{prop}[thm]{Proposition}
\newtheorem{cor}[thm]{Corollary}

\theoremstyle{definition}
\newtheorem{df}[thm]{Definition}
\newtheorem{rem}[thm]{Remark}

\numberwithin{equation}{section}

\renewcommand{\phi}{\varphi}

\newcommand{\ep}{\varepsilon}

\newcommand{\Ad}{\operatorname{Ad}}
\newcommand{\aInn}{\overline{\operatorname{Inn}}}
\newcommand{\Aut}{\operatorname{Aut}}
\newcommand{\Coker}{\operatorname{Coker}}
\newcommand{\Ext}{\operatorname{Ext}}
\newcommand{\HInn}{\operatorname{HInn}}
\newcommand{\Homeo}{\operatorname{Homeo}}
\newcommand{\id}{\operatorname{id}}
\newcommand{\Lip}{\operatorname{Lip}}
\newcommand{\Tor}{\operatorname{Tor}}
\newcommand{\tr}{\operatorname{tr}}

\newcommand{\N}{\mathbb{N}}
\newcommand{\Z}{\mathbb{Z}}
\newcommand{\Q}{\mathbb{Q}}
\newcommand{\R}{\mathbb{R}}
\newcommand{\C}{\mathbb{C}}
\newcommand{\T}{\mathbb{T}}

\title{$\mathcal{Z}$-stability of crossed products 
by strongly outer actions}
\author{Hiroki Matui \\
Graduate School of Science \\
Chiba University \\
Inage-ku, Chiba 263-8522, Japan 
\and
Yasuhiko Sato \\
Department of Mathematics \\
Hokkaido University \\
Kita-ku, Sapporo 060-0810, Japan}
\date{}

\begin{document}
\maketitle

\begin{abstract}
We consider a certain class of 
unital simple stably finite $C^*$-algebras 
which absorb the Jiang-Su algebra $\mathcal{Z}$ tensorially. 
Under a mild assumption, 
we show that the crossed product of a $C^*$-algebra in this class 
by a strongly outer action of $\Z^N$ or a finite group 
is $\mathcal{Z}$-stable. 
As an application, we also prove that 
all strongly outer actions of $\Z^2$ on $\mathcal{Z}$ are 
mutually cocycle conjugate. 
\end{abstract}

\section{Introduction}

The Jiang-Su algebra $\mathcal{Z}$, 
which was introduced by Jiang and Su in \cite{JS}, 
is a unital, simple, separable, stably finite and nuclear $C^*$-algebra 
$KK$-equivalent to $\C$. 
In Elliott's program to classify nuclear $C^*$-algebras 
via $K$-theoretic invariants 
(see \cite{Rtext} for an introduction to this subject), 
the Jiang-Su algebra plays a central role. 
More precisely, 
the recent progress of the program \cite{Winter0708,LN1} tells us that 
one can only expect $K$-theoretic classification results 
up to $\mathcal{Z}$-stability, 
where a $C^*$-algebra is called $\mathcal{Z}$-stable 
if it absorbs $\mathcal{Z}$ tensorially. 
One may view $\mathcal{Z}$ as being the stably finite analogue 
of the Cuntz algebra $\mathcal{O}_\infty$. 

It is then natural to study permanence of $\mathcal{Z}$-stability 
under taking crossed products. 
I. Hirshberg and W. Winter \cite{HW} showed that 
$\mathcal{D}$-stability passes to crossed products 
by $\Z$, $\R$ or compact groups 
provided the group action has a Rohlin property, 
where $\mathcal{D}$ is 
a $K_1$-injective strongly self-absorbing $C^*$-algebra. 
The second-named author \cite{S2,S3} showed 
$\mathcal{Z}$-stability of the crossed product $A\rtimes_\alpha\Z$ 
for certain projectionless $C^*$-algebras with unique trace and 
automorphisms whose extension to the weak closure 
in the tracial representation are aperiodic. 
The present paper generalizes this result and 
gives a positive solution to the question of permanence 
under taking crossed products in the following setting. 
Let $A$ be a unital, simple, separable and $\mathcal{Z}$-stable $C^*$-algebra 
such that $A\otimes B$ has tracial rank zero for any UHF algebra $B$. 
Let $G$ be $\Z^N$ or a finite group. 
An action $\alpha:G\curvearrowright A$ is said to be strongly outer 
if its extension to the weak closure in any tracial representation is outer 
(see Definition \ref{so&wR} for the precise definition). 
Under certain technical assumptions, 
we show that any strongly outer action $\alpha$ has the weak Rohlin property 
(Theorem \ref{wRohlintype}). 
The weak Rohlin property, which was first introduced in \cite{S3}, 
is a variant of the Rohlin property (Definition \ref{so&wR}). 
Usually the definition of the (tracial) Rohlin property 
for single automorphisms or group actions involves projections. 
However the $C^*$-algebra $A$ may not contain enough projections, 
and so we have to weaken the definition, that is, 
we use positive elements in place of projections. 
Then it will be shown that 
when $A$ is further assumed to be nuclear and 
satisfy the universal coefficient theorem, 
any action $\alpha:G\curvearrowright A$ with the weak Rohlin property is 
cocycle conjugate to $\alpha\otimes\id:G\curvearrowright A\otimes\mathcal{Z}$ 
(Corollary \ref{absorbZ0}). 
In particular, 
the crossed product $A\rtimes_\alpha G$ is $\mathcal{Z}$-stable. 
It is also natural to ask 
if the class of $C^*$-algebras under consideration is closed 
under taking crossed products by strongly outer actions of 
$\Z$ or a finite group. 
We give partial answers to this question 
(Theorems \ref{closedbyfnt}, \ref{closedbyZ}). 
With the aid of the classification theorem \cite{Winter0708,LN1}, 
we also prove the following (Corollary \ref{AHclosedbyZ}): 
Let $A$ be a unital simple AH algebra 
with real rank zero and slow dimension growth 
and let $\alpha:\Z\curvearrowright A$ be a strongly outer action. 
Suppose that $A$ has a unique trace and 
$\alpha_k$ induces the identity on $K_i(A)\otimes\Q$ for some $k\in\N$. 
Then the crossed product $A\rtimes_\alpha\Z$ is again 
a unital simple AH algebra with real rank zero and slow dimension growth. 

In the latter half of the paper, 
we study $\Z^2$-actions on the Jiang-Su algebra. 
Classification of group actions is 
one of the most fundamental subjects in the theory of operator algebras. 
For AFD factors, 
a complete classification is known for actions of countable amenable groups. 
However, classification of automorphisms or group actions 
on $C^*$-algebras is still a far less developed subject, 
partly because of $K$-theoretical difficulties. 
We briefly review classification results of automorphisms or $\Z^N$-actions 
known so far. 
For AF and AT algebras, 
A. Kishimoto \cite{K95crelle,K96JFA,K98JOT,K98JFA} showed 
the Rohlin property for a certain class of automorphisms and 
obtained a cocycle conjugacy result. 
The first-named author \cite{M09} extended this result to 
unital simple AH algebras with real rank zero and slow dimension growth. 
The second-named author \cite{S3} proved that 
strongly outer $\Z$-actions on $\mathcal{Z}$ are 
unique up to cocycle conjugacy. 
As for $\Z^N$-actions, 
H. Nakamura \cite{N1} introduced the notion of the Rohlin property and 
classified product type actions of $\Z^2$ on UHF algebras. 
T. Katsura and the first-named author \cite{KM} gave 
a complete classification of uniformly outer $\Z^2$-actions on UHF algebras 
by using the Rohlin property, 
and then this result was extended to a certain class of 
uniformly outer $\Z^2$-actions on unital simple AF algebras \cite{M09}. 
For Kirchberg algebras, 
complete classification of aperiodic automorphisms was given 
by H. Nakamura \cite{N2}. 
M. Izumi and the first-named author \cite{IM} classified 
a large class of $\Z^2$-actions 
and also showed the uniqueness of $\Z^N$-actions 
on the Cuntz algebras $\mathcal{O}_2$ and $\mathcal{O}_\infty$ 
(see also \cite{M08}). 

Our main result in the latter half of this paper is 
the uniqueness of strongly outer $\Z^2$-actions 
on the Jiang-Su algebra $\mathcal{Z}$ 
up to strong cocycle conjugacy (Theorem \ref{uniqueZ2}). 
The proof needs three ingredients. 
The first one is 
the absorption of the trivial action up to cocycle conjugacy, 
which is obtained in the first half of the paper. 
The second is 
the uniqueness of $\Z^2$-actions on UHF algebras of infinite type, 
which was shown in \cite{KM}. 
The last one is a lemma concerning certain homotopies of unitaries 
(Lemma \ref{3torus}). 
In order to prove it, we discuss 
when an almost commuting triple of unitary matrices is approximated 
by a commuting triple of unitary matrices (Lemma \ref{GongLin}). 
In the last section 
we determine when a strongly outer cocycle action of $\Z^2$ 
on UHF algebras or $\mathcal{Z}$ is equivalent to a genuine action.

\section{Preliminaries}

The cardinality of a set $F$ is written $\lvert F\rvert$. 
We let $\log$ be the standard branch 
defined on the complement of the negative real axis.
For a Lipschitz continuous function $f$, 
we denote its Lipschitz constant by $\Lip(f)$. 
The normalized trace on $M_n$ is written $\tr$. 

Let $A$ be a $C^*$-algebra. 
For $a,b\in A$, we mean by $[a,b]$ the commutator $ab-ba$. 
The set of tracial states on $A$ is denoted by $T(A)$. 
For $a\in A$, we define 
\[
\lVert a\rVert_2=\sup_{\tau\in T(A)}\lVert a\rVert_\tau, 
\]
where $\lVert a\rVert_\tau=\tau(a^*a)^{1/2}$. 
If $A$ is simple and $T(A)$ is non-empty, 
then $\lVert\cdot\rVert_2$ is a norm. 
For $\tau\in T(A)$, 
we let $(\pi_\tau,H_\tau)$ denote 
the GNS representation of $A$ associated with $\tau$. 
When $A$ is unital, 
we mean by $U(A)$ the set of all unitaries in $A$. 
The connected component of the identity in $U(A)$ is denoted by $U(A)_0$. 
For $\tau\in T(A)$, 
the de la Harpe-Skandalis determinant associated with $\tau$ is 
written $\Delta_\tau:U(A)_0\to\R/\tau(K_0(A))$ (\cite[Section 1]{HS}). 
We frequently use the following fact: 
when $u,v\in U(A)$ satisfy $\lVert u-1\rVert+\lVert v-1\rVert<2$, 
one has $\tau(\log(uv))=\tau(\log u)+\tau(\log v)$ (see \cite[Lemma 1]{HS}). 
For $u\in U(A)$, 
the inner automorphism induced by $u$ is written $\Ad u$. 
An automorphism $\alpha\in\Aut(A)$ is called outer 
when it is not inner. 
For $\alpha\in\Aut(A)$, 
we let $T(A)^\alpha=\{\tau\in T(A)\mid\tau\circ\alpha=\tau\}$. 
A single automorphism is often identified 
with the $\Z$-action generated by it. 
Let $\alpha:G\curvearrowright A$ be 
an action of a discrete group $G$ on a unital $C^*$-algebra $A$. 
When $\alpha_g$ is outer for all $g\in G$ except for the neutral element, 
the action $\alpha$ is said to be outer. 
We let $T(A)^\alpha=\bigcap_{g\in G}T(A)^{\alpha_g}$. 
The fixed point subalgebra of $A$ is $A^\alpha$. 
When $\phi$ is a homomorphism between $C^*$-algebras, 
$K_0(\phi)$ and $K_1(\phi)$ mean the induced homomorphisms on $K$-groups. 

Let $A$ be a separable $C^*$-algebra. 
Set 
\[
c_0(A)=\{(a_n)\in\ell^\infty(\N,A)\mid
\lim_{n\to\infty}\lVert a_n\rVert=0\},\quad 
A^\infty=\ell^\infty(\N,A)/c_0(A). 
\]
We identify $A$ with the $C^*$-subalgebra of $A^\infty$ 
consisting of equivalence classes of constant sequences. 
We let 
\[
A_\infty=A^\infty\cap A'
\]
and call it the central sequence algebra of $A$. 
A sequence $(x_n)_n\in\ell^\infty(\N,A)$ is called a central sequence 
if $\lVert[a,x_n]\rVert\to0$ as $n\to\infty$ for all $a\in A$. 
A central sequence is a representative of an element in $A_\infty$. 
When $\alpha$ is an automorphism of $A$ or 
an action of a discrete group on $A$, 
we can consider its natural extension on $A^\infty$ and $A_\infty$. 
We denote it by the same symbol $\alpha$. 

We set up some terminology for group actions. 

\begin{df}
Let $\alpha:G\curvearrowright A$ and $\beta:G\curvearrowright B$ 
be actions of a discrete group $G$ on unital $C^*$-algebras $A$ and $B$. 
\begin{enumerate}
\item The two actions $\alpha$ and $\beta$ are said to be conjugate 
when there exists an isomorphism $\mu:A\to B$ such that 
$\alpha_g=\mu^{-1}\circ\beta_g\circ\mu$ for all $g\in G$. 
\item A family of unitaries $\{u_g\}_{g\in G}$ in $A$ is called 
an $\alpha$-cocycle 
if one has $u_g\alpha_g(u_h)=u_{gh}$ for all $g,h\in G$. 
When $\{u_g\}_g$ is an $\alpha$-cocycle, 
the perturbed action $\alpha^u:G\curvearrowright A$ is 
defined by $\alpha^u_g=\Ad u_g\circ\alpha_g$. 
\item The two actions $\alpha$ and $\beta$ are said to be cocycle conjugate 
if there exists an $\alpha$-cocycle $\{u_g\}_{g\in G}$ in $A$ such that 
$\alpha^u$ is conjugate to $\beta$. 
\item The two actions $\alpha$ and $\beta$ are said to be 
strongly cocycle conjugate 
if there exist an $\alpha$-cocycle $\{u_g\}_{g\in G}$ in $A$ 
and a sequence of unitaries $\{v_n\}_{n=1}^\infty$ in $A$ such that 
$\alpha^u$ is conjugate to $\beta$ and 
$\lim_{n\to\infty}\lVert u_g-v_n\alpha_g(v_n^*)\rVert=0$ for all $g\in G$. 
We remark that if $\alpha$ and $\beta$ are strongly cocycle conjugate, 
then they are approximately conjugate 
in the sense of \cite[Definition 7 (1)]{N1}. 
\end{enumerate}
\end{df}

\begin{df}\label{ca}
Let $A$ be a unital $C^*$-algebra and let $G$ be a discrete group. 
\begin{enumerate}
\item A pair $(\alpha,u)$ of 
a map $\alpha:G\to\Aut(A)$ and a map $u:G\times G\to U(A)$ 
is called a cocycle action of $G$ on $A$ 
if 
\[
\alpha_g\circ\alpha_h=\Ad u(g,h)\circ\alpha_{gh}
\]
and 
\[
u(g,h)u(gh,k)=\alpha_g(u(h,k))u(g,hk)
\]
hold for any $g,h,k\in G$. 
Notice that $\alpha$ gives rise to a (genuine) action of $G$ on $A_\infty$. 
\item A cocycle action $(\alpha,u)$ is said to be outer 
if $\alpha_g$ is outer for every $g\in G$ except for the neutral element. 
\item Two cocycle actions $(\alpha,u)$ and $(\beta,v)$ of $G$ on $A$ are 
said to be equivalent 
if there exists a map $w:G\to U(A)$ such that 
\[
\alpha_g=\Ad w(g)\circ\beta_g
\]
and 
\[
u(g,h)=w(g)\beta_g(w(h))v(g,h)w(gh)^*
\]
for every $g,h\in G$. 
\end{enumerate}
\end{df}

We denote the Jiang-Su algebra by $\mathcal{Z}$ and 
the unique trace on $\mathcal{Z}$ by $\omega$. 
X. Jiang and H. Su proved in \cite[Theorem 3, Theorem 4]{JS} that 
any unital endomorphisms of $\mathcal{Z}$ is approximately inner and that 
$\mathcal{Z}$ is isomorphic to 
the infinite tensor product of copies of itself. 
In particular, 
$\mathcal{Z}$ is strongly self-absorbing, cf. \cite[Definition 1.3]{TW07TAMS}. 
These facts are freely used in this article. 
When a $C^*$-algebra $A$ satisfies $A\cong A\otimes\mathcal{Z}$, 
we say that $A$ absorbs $\mathcal{Z}$ tensorially, 
or $A$ is $\mathcal{Z}$-stable. 

We let $Q$ denote the universal UHF algebra, that is, 
$Q$ is the UHF algebra satisfying $K_0(Q)=\Q$. 

We recall the definition of tracial rank zero introduced by H. Lin. 

\begin{df}[{\cite{L01TAMS,L01PLMS}}]
A unital simple $C^*$-algebra $A$ is said to have tracial rank zero 
if for any finite subset $F\subset A$, any $\ep>0$ and 
any non-zero positive element $x\in A$ 
there exists a finite dimensional subalgebra $B\subset A$ with $p=1_B$ 
satisfying the following. 
\begin{enumerate}
\item $\lVert[a,p]\rVert<\ep$ for all $a\in F$. 
\item The distance from $pap$ to $B$ is less than $\ep$ for all $a\in F$. 
\item $1-p$ is unitarily equivalent to a projection in $\overline{xAx}$. 
\end{enumerate}
\end{df}

In \cite{L04Duke}, H. Lin gave a classification theorem 
for unital separable simple nuclear $C^*$-algebras 
with tracial rank zero which satisfy the UCT 
(\cite[Theorem 5.2]{L04Duke}). 
Indeed, the class of such $C^*$-algebras agrees with 
the class of all unital simple AH algebras 
with real rank zero and slow dimension growth. 

The following lemma is a variant of \cite[Theorem 3.6]{LN2}. 

\begin{lem}\label{onlyQ}
Let $A$ be a unital simple separable and 
approximately divisible $C^*$-algebra. 
If $A\otimes Q$ has tracial rank zero, then so does $A$. 
\end{lem}
\begin{proof}
Suppose that we are given a finite subset $F\subset A$, $\ep>0$ and 
a non-zero positive element $x\in A$. 
We can find mutually orthogonal non-zero positive elements $x_1,x_2$ 
in $\overline{xAx}$, because $A$ is simple and infinite dimensional. 
Simplicity of $A$ also implies that 
there exists $k\in\N$ such that 
for any $n\geq k$ and a minimal projection $e$ in $M_n$, 
$1\otimes e$ is Murray-von Neumann equivalent to a projection 
in the hereditary subalgebra generated by $x_1\otimes1$ in $A\otimes M_n$. 
Since $A\otimes Q$ has tracial rank zero, 
there exists a finite dimensional subalgebra $B\subset A\otimes Q$ 
with $p=1_B$ satisfying the following. 
\begin{itemize}
\item $\lVert[a\otimes1,p]\rVert<\ep$ for all $a\in F$. 
\item The distance from $p(a\otimes1)p$ to $B$ is less than $\ep$ 
for all $a\in F$. 
\item $1-p$ is Murray-von Neumann equivalent to a projection 
in the hereditary subalgebra generated by $x_2\otimes 1$ in $A\otimes Q$. 
\end{itemize}
Clearly we may assume that there exists $m\geq k$ such that 
$B$ is a subalgebra of $A\otimes M_m$ and 
$1-p$ is Murray-von Neumann equivalent to a projection 
in the hereditary subalgebra generated by $x_2\otimes 1$ in $A\otimes M_m$. 
Set $D=A\otimes(M_m\oplus M_{m+1})$. 
Define a homomorphism $\phi:B\to D$ 
by $\phi(b)=b\oplus b$. 
One can easily verify the following. 
\begin{itemize}
\item $\lVert[a\otimes1,\phi(p)]\rVert<\ep$ for all $a\in F$. 
\item The distance from $\phi(p)(a\otimes1)\phi(p)$ to $\phi(B)$ is 
less than $\ep$ for all $a\in F$. 
\item $1_D-\phi(p)$ is Murray-von Neumann equivalent to a projection 
in the hereditary subalgebra generated by $x\otimes 1$ in $D$. 
\end{itemize}
As $A$ is approximately divisible, one can find 
a unital homomorphism from $M_m\oplus M_{m+1}$ to $A_\infty$ 
(see \cite[Lemma 3.1]{M09} for example). 
It follows that there exists a unital homomorphism $\pi:D\to A^\infty$ 
such that $\pi(a\otimes1)=a$ for any $a\in A$. 
The finite dimensional subalgebra $\pi(\phi(B))\subset A^\infty$ 
lifts to a finite dimensional subalgebra of $A$, 
and so the proof is completed. 
\end{proof}

We introduce two classes $\mathcal{C}$ and $\mathcal{C}_0$ of $C^*$-algebras 
as follows. 

\begin{df}
Let $\mathcal{C}$ be 
the class of unital simple separable $C^*$-algebras $A$ such that 
$A\otimes\mathcal{Z}\cong A$ and $A\otimes Q$ has tracial rank zero. 
Let $\mathcal{C}_0$ be the family of all $C^*$-algebras $A\in\mathcal{C}$ 
which are nuclear and satisfy the universal coefficient theorem. 
\end{df}

\begin{rem}\label{RemonC0}
\begin{enumerate}
\item By Lemma \ref{onlyQ}, if $A$ is in $\mathcal{C}$, then 
$A\otimes B$ has tracial rank zero for any UHF algebra $B$. 
\item A unital simple separable nuclear $C^*$-algebra with tracial rank zero 
satisfying the UCT is in $\mathcal{C}_0$, 
because it absorbs $\mathcal{Z}$ tensorially 
by \cite[Corollary 3.1]{TW08CJM}. 
The Jiang-Su algebra $\mathcal{Z}$ itself is in $\mathcal{C}_0$ 
by \cite[Theorem 4, Theorem 5]{JS}. 
\item The class $\mathcal{C}_0$ is the same as that 
treated in the classification theorems 
due to W. Winter \cite[Theorem 7.1]{Winter0708}, 
H. Lin and Z. Niu \cite[Theorem 5.4]{LN1} 
(see also \cite[Theorem 5.1]{TW0903}). 
Namely, for any $A,B\in\mathcal{C}_0$ and 
any graded ordered isomorphism $\phi:K_*(A)\to K_*(B)$, 
there exists an isomorphism $\Phi:A\to B$ inducing $\phi$. 
\item By \cite[Corollary 8.6]{Winter0708}, 
the following three conditions are equivalent for $A\in\mathcal{C}_0$: 
(i) $A$ has tracial rank zero, (ii) $A$ has real rank zero, 
(iii) $K_0(A)$ has Riesz interpolation and $K_0(A)/\Tor(K_0(A))\neq\Z$. 
\item If $A$ is a unital simple ASH algebra 
whose projections separate traces, then 
$A\otimes B$ has tracial rank zero for any UHF algebra $B$. 
In fact, by \cite[Corollary 2.2]{NgWinter}, 
$A$ has locally finite decomposition rank. 
By \cite[Theorem 1.4 (e)]{BKR}, $A\otimes B$ has real rank zero 
since projections separate traces. 
Then \cite[Theorem 2.1]{Winter07JFA} tells us that 
$A\otimes B$ has tracial rank zero. 
We remark that $A$ satisfies the UCT by \cite[Corollary 2.1]{NgWinter}. 
\item In \cite{W1006,Toms0910} it is shown that 
a unital simple ASH algebra $A$ has slow dimension growth 
if and only if $A$ is $\mathcal{Z}$-stable. 
It is also known that 
if $A$ is a unital simple separable $C^*$-algebra 
with finite decomposition rank, then 
$A\cong A\otimes\mathcal{Z}$ (\cite[Theorem 5.1]{Winter0806}). 
\item Let $\alpha\in\Homeo(X)$ be a minimal homeomorphism 
of an infinite, compact, metrizable, finite dimensional space $X$ and 
let $A=C(X)\rtimes_\alpha\Z$. 
By Theorem 0.2 of \cite{TW0903}, $A$ is $\mathcal{Z}$-stable. 
By Proposition 5.3 and Proposition 5.7 of \cite{TW0903}, 
if projections in $A$ separate traces, then 
$A\otimes B$ has tracial rank zero for any UHF algebra $B$ of infinite type, 
and hence $A$ is in $\mathcal{C}_0$. 
\item When $A$ is in $\mathcal{C}$ and $\tau\in T(A)$ is extremal, 
it is easy to see that 
$\pi_\tau(A)''$ is the hyperfinite II$_1$-factor 
(see \cite[Lemma 2.16]{OP2}). 
\end{enumerate}
\end{rem}

Next we introduce the notions of strong outerness and 
the weak Rohlin property for group actions on $C^*$-algebras. 
Let $G$ be a discrete group. 
For a finite subset $F\subset G$ and $\ep>0$, 
we say that a finite subset $K\subset G$ is $(F,\ep)$-invariant 
if $\lvert K\cap\bigcap_{g\in F}g^{-1}K\rvert\geq(1{-}\ep)\lvert K\rvert$. 

\begin{df}\label{so&wR}
Let $A$ be a unital simple $C^*$-algebra with $T(A)$ non-empty. 
\begin{enumerate}
\item We say that an automorphism $\alpha\in\Aut(A)$ 
is not weakly inner for $\tau\in T(A)^\alpha$ 
if the weak extension of $\alpha$ to an automorphism of $\pi_\tau(A)''$ 
is outer. 
\item An action $\alpha:G\curvearrowright A$ of a discrete group $G$ on $A$ 
is said to be strongly outer 
if $\alpha_g$ is not weakly inner 
for any $\tau\in T(A)^{\alpha_g}$ and $g\in G\setminus\{e\}$. 
Likewise, a cocycle action $(\alpha,u)$ of $G$ on $A$ 
is said to be strongly outer 
if $\alpha_g$ is not weakly inner 
for any $\tau\in T(A)^{\alpha_g}$ and $g\in G\setminus\{e\}$. 
\item Assume $G=\Z^N$. 
We say that $\alpha$ has the weak Rohlin property 
if for any finite subset $F\subset G$ and $\ep>0$ 
there exists an $(F,\ep)$-invariant finite subset $K\subset G$ and 
a central sequence $(f_n)_n$ in $A$ such that $0\leq f_n\leq1$, 
\[
\lim_{n\to\infty}\lVert\alpha_g(f_n)\alpha_h(f_n)\rVert=0
\]
for all $g,h\in K$ with $g\neq h$ and 
\[
\lim_{n\to\infty}
\max_{\tau\in T(A)}\lvert\tau(f_n)-\lvert K\rvert^{-1}\rvert=0. 
\]
When $G=\Z$, in addition to the conditions above, 
we further impose the restriction that 
$K$ is of the form $\{0,1,\dots,k\}$ for some $k\in\N$. 
\item Suppose that $G$ is a finite group. 
We say that $\alpha$ has the weak Rohlin property 
if there exists a central sequence $(f_n)_n$ in $A$ such that 
the properties in (3) hold for $G$ in place of $K$. 
\item For a cocycle action $(\alpha,u)$, we can define 
the weak Rohlin property in the same way, 
because $\alpha$ induces a genuine action on $A_\infty$. 
\end{enumerate}
\end{df}

The projection free tracial Rohlin property introduced 
in \cite[Definition 2.7]{A} is 
quite similar to the weak Rohlin property defined above. 
But we do not know if either implies the other in general. 

\begin{rem}\label{wR>so}
One can show that the weak Rohlin property implies strong outerness 
in a similar fashion to \cite[Lemma 4.4]{K96JFA}. 
In fact, if $\alpha:G\curvearrowright A$ has the weak Rohlin property, 
then for any $g\in G\setminus\{e\}$, $a\in A$ and $\ep>0$ 
there exist $f_1,f_2,\dots,f_k$ in $A$ such that 
\[
0\leq f_i\leq1,\quad \lVert[f_i,a]\rVert\approx0,\quad 
\lVert f_ia\alpha_g(f_i)\rVert\approx0,\quad 
f_if_j\approx0
\]
for all $i\neq j$ and 
$\tau(f_1+f_2+\dots+f_k)>1-\ep$ for all $\tau\in T(A)$. 
Hence, by the same argument as in \cite[Lemma 4.3]{K96JFA}, 
we get $\phi(a\lambda)=0$ for any $\phi\in T(A\rtimes_{\alpha_g}\Z)$, 
where $\lambda\in A\rtimes_{\alpha_g}\Z$ is the implementing unitary 
of the automorphism $\alpha_g$. 
It follows from the proof of \cite[Lemma 4.4]{K96JFA} that 
$\alpha_g$ is not weakly inner for any $\tau\in T(A)^{\alpha_g}$. 
Therefore $\alpha$ is strongly outer. 
This observation also shows that 
$T(A)^\alpha$ is canonically isomorphic to $T(A\rtimes_\alpha G)$ 
via the inclusion of $A$ into $A\rtimes_\alpha G$. 
\end{rem}

\begin{rem}\label{sotimesid}
Let $\alpha:G\curvearrowright A$ be a strongly outer action and 
let $B$ be a unital simple $C^*$-algebra with a unique trace. 
Then one can easily see that 
$\alpha\otimes\id:G\curvearrowright A\otimes B$ is strongly outer, too. 
\end{rem}

\begin{rem}\label{so=uo}
In Section 4 and 5 of \cite{M09} it is shown that 
uniformly outer actions of $\Z$ or $\Z^2$ on certain AH algebras have 
the (tracial) Rohlin property 
under some technical assumptions. 
By \cite[Lemma 4.4]{K96JFA} uniform outerness implies strong outerness, 
and the proofs in \cite{M09} only use strong outerness of the actions. 
Hence one can replace uniform outerness with strong outerness 
in the results of \cite{M09}. 
In particular, for a $\Z$-action $\alpha$ stated in \cite[Theorem 4.8]{M09}, 
the following three conditions are equivalent: 
(i) $\alpha$ is uniformly outer, 
(ii) $\alpha$ is strongly outer, 
(iii) $\alpha$ has the Rohlin property. 
In the same way, for a $\Z^2$-action $\alpha$ 
as in Corollary 5.6 or Corollary 5.7 of \cite{M09} 
(including $\alpha$ on a UHF algebra), 
the three conditions above are equivalent, too. 
\end{rem}

\section{A Rohlin type theorem}

In this section we establish a Rohlin type theorem 
for actions of $\Z^N$ and a finite group on $C^*$-algebras in $\mathcal{C}$ 
(Theorem \ref{wRohlintype}). 
To this end 
we first need to recall the inductive limit construction of $\mathcal{Z}$ 
from \cite{S1,JS}. 
As mentioned in Section 2, 
$\omega$ denotes the unique tracial state on $\mathcal{Z}$. 

For natural numbers $p,q\in\N$, we let 
\[
I(p,q)=\{f\in C([0,1],M_p\otimes M_q)\mid
f(0)\in M_p\otimes\C,\ f(1)\in\C\otimes M_q\}. 
\]
Let $(p_n)_n$ and $(q_n)_n$ be increasing sequences 
of natural numbers such that 
\begin{itemize}
\item $p_n$ and $q_n$ are relatively prime. 
\item $p_n$ divides $p_{n+1}$ and $q_n$ divides $q_{n+1}$. 
\item $p_{n+1}$ and $q_{n+1}$ are greater than $2p_nq_n$. 
\end{itemize}
Let $C_n=C([0,1],M_{p_n}\otimes M_{q_n})$ and 
let $B_n\subset C_n$ be the subalgebra consisting of constant functions 
(i.e. $B_n\cong M_{p_n}\otimes M_{q_n}$). 
It was proved in \cite[Section 2]{S1} that 
there exist a unital injective homomorphism $\psi_n:C_n\to C_{n+1}$ and 
a unital subalgebra $A_n\subset C_n$ such that the following hold. 
(The reader should be warned that 
$\phi_n$ of \cite[Section 2]{S1} is not the same as $\psi_n$ used here, 
but is equal to $\Ad u\circ\psi_n$ for some unitary $u\in C_{n+1}$. )
\begin{itemize}
\item $\psi_n(A_n)\subset A_{n+1}$ and $\psi_n(B_n)\subset B_{n+1}$. 
\item $A_n$ is isomorphic to $I(p_n,q_n)$ and 
contains $\{f\in C_n\mid f(0)=f(1)=0\}$. 
\item The inductive limit $C^*$-algebra constructed from 
$\psi_n:A_n\to A_{n+1}$ is isomorphic to the Jiang-Su algebra $\mathcal{Z}$. 
\item Let $\tilde\psi_n:A_n\to\mathcal{Z}$ be the canonical inclusion and 
let $\nu_n$ be the probability measure on $[0,1]$ satisfying 
\[
\omega(\tilde\psi_n(f))=\int_0^1\tr(f(t))\,d\nu_n(t)\quad 
\]
for all $f\in A_n$. 
Then $\nu_n$ has no atoms in $\{0,1\}$ (\cite[Lemma 2.2]{S1}). 
\end{itemize}
We let $B$ be the inductive limit of $\psi_n:B_n\to B_{n+1}$. 
Of course, $B$ is a UHF algebra. 
In what follows, 
we omit $\psi_n$ and regard $C_n$ as a subalgebra of $C_{n+1}$. 
Define $\omega_n\in T(C_n)$ by 
\[
\omega_n(f)=\int_0^1\tr(f(t))\,d\nu_n(t)
\]
for $f\in C_n$. 
Clearly $\omega_n$ is an extension of $\omega|A_n$ to $C_n$. 
It is also easy to see that $\omega_{n+1}|C_n$ equals $\omega_n$. 

\begin{prop}\label{central}
Let $A$ be a unital simple separable $C^*$-algebra 
such that $A\otimes B$ has tracial rank zero. 
For any finite subset $F\subset A$ and $\ep>0$, 
there exist a finite subset $G\subset A\otimes\mathcal{Z}$ and $\delta>0$ 
satisfying the following. 
If $x\in A\otimes\mathcal{Z}$ is a positive element such that 
$\lVert x\rVert\leq1$ and 
\[
\lVert[x,a]\rVert_2<\delta
\]
for all $a\in G$, 
then there exists a positive element $y\in A\otimes\mathcal{Z}$ such that 
$\lVert y\rVert\leq1$, $\lVert x-y\rVert_2<\ep$ and 
\[
\lVert[y,a\otimes1]\rVert<\ep
\]
for all $a\in F$. 
\end{prop}
\begin{proof}
We may assume that $F$ is contained in the unit ball of $A$. 
Use tracial rank zero to find a projection $e\in A\otimes B$ and 
a finite dimensional unital subalgebra $E\subset e(A\otimes B)e$ 
such that the following are satisfied. 
\begin{itemize}
\item For any $a\in F$, $\lVert[a\otimes1,e]\rVert<\ep/4$. 
\item For any $a\in F$, the distance from $e(a\otimes1)e$ to $E$ is 
less than $\ep/4$. 
\item $\lVert1-e\rVert_2<\ep/4$. 
\end{itemize}
We may assume that there exists $n\in\N$ such that 
$e$ is in $A\otimes B_n$ and $E$ is a subalgebra of $A\otimes B_n$. 
Choose a real valued continuous function $f\in C([0,1])$ so that 
\[
f(0)=f(1)=0,\quad 0\leq f\leq 1\quad\text{and}\quad
\int_0^1(1-f(t))^2\,d\nu_n(t)<\ep^2/16. 
\]
We regard $f$ as an element in $A_n$. 
Set 
\[
K=\{b(1\otimes f)\mid b\in E,\ \lVert b\rVert=1\}\subset A\otimes A_n
\subset A\otimes\mathcal{Z}. 
\]
Since $K$ is compact, 
there exist a finite subset $G\subset K$ and $\delta>0$ 
such that the following holds: 
If a positive element $x\in A\otimes\mathcal{Z}$ satisfies 
$\lVert x\rVert\leq1$ and 
$\lVert[x,a]\rVert_2<\delta$ for all $a\in G$, then 
$\lVert[x,a]\rVert_2$ is less than $\ep/4$ for all $a\in K$. 
Suppose that 
we are given a positive element $x\in A\otimes\mathcal{Z}$ such that 
$\lVert x\rVert\leq1$ and 
$\lVert[x,a]\rVert_2$ is less than $\delta$ for all $a\in G$. 
We may assume that there exists $m\geq n$ such that $x\in A\otimes A_m$. 
For any $\tau\in T(A)$ and $b\in E$ with $\lVert b\rVert=1$, 
one has 
\begin{align*}
\lVert[x,b]\rVert_{\tau\otimes\omega_m}
&\leq\lVert[x,b(1\otimes f)]\rVert_{\tau\otimes\omega_m}
+2\lVert b-b(1\otimes f)\rVert_{\tau\otimes\omega_m}\\
&\leq\lVert[x,b(1\otimes f)]\rVert_{\tau\otimes\omega}
+2\lVert1-1\otimes f\rVert_{\tau\otimes\omega_n}\\
&<\ep/4+\ep/2=3\ep/4. 
\end{align*}
Using the Haar measure on the compact group $U(E)$, 
we define a positive element $y\in e(A\otimes C_m)e$ by 
\[
y=\int_{U(E)}uxu^*\,du. 
\]
Since $y$ commutes with every element in $E$, we obtain 
$\lVert[y,a\otimes1]\rVert<\ep$ for all $a\in F$. 
Furthermore, 
\[
\lVert x-y\rVert_{\tau\otimes\omega_m}
\leq\lVert x-xe\rVert_{\tau\otimes\omega_m}
+\lVert xe-y\rVert_{\tau\otimes\omega_m}
<\ep/4+3\ep/4=\ep
\]
for any $\tau\in T(A)$. 
Choose a real valued continuous function $g\in C([0,1])$ so that 
\[
g(0)=g(1)=0,\quad 0\leq g\leq 1\quad\text{and}\quad
\int_0^1(1-g(t))^2\,d\nu_m(t)\approx0. 
\]
and regard $g$ as an element in the center of $C_m$. 
Then $y(1\otimes g)$ is in $A\otimes A_m$ and 
$\lVert y-y(1\otimes g)\rVert_{\tau\otimes\omega_m}\approx0$ 
for all $\tau\in T(A)$. 
By replacing $y$ with $y(1\otimes g)$, 
we can complete the proof. 
\end{proof}

As a direct consequence of Proposition \ref{central}, 
we get the following. 

\begin{prop}\label{central2}
Let $A\in\mathcal{C}$. 
For any finite subset $F\subset A$ and $\ep>0$, 
there exist a finite subset $G\subset A$ and $\delta>0$ such that 
the following hold. 
If $x\in A$ is a positive element such that 
$\lVert x\rVert\leq1$ and 
\[
\lVert[x,a]\rVert_2<\delta
\]
for any $a\in G$, 
then there exists a positive element $y\in A$ such that 
$\lVert y\rVert\leq1$, $\lVert x-y\rVert_2<\ep$ and 
\[
\lVert[y,a]\rVert<\ep
\]
for any $a\in F$. 
\end{prop}

\begin{prop}\label{central3}
Let $A$ be in $\mathcal{C}$ 
and let $\Gamma\subset\Aut(A)$ be a finite subset containing the identity. 
Suppose that there exists a sequence $(e_n)_n$ in $A$ 
such that $0\leq e_n\leq1$ and the following hold. 
\begin{enumerate}
\item $\lVert\gamma(e_n)\gamma'(e_n)\rVert_2\to0$ 
for any $\gamma,\gamma'\in\Gamma$ such that $\gamma\neq\gamma'$. 
\item For every $a\in A$, we have 
$\lVert[a,e_n]\rVert_2\to0$. 
\end{enumerate}
Then there exists a central sequence $(f_n)_n$ in $A$ 
such that $0\leq f_n\leq1$ and the following hold. 
\begin{enumerate}
\item $\lVert\gamma(f_n)\gamma'(f_n)\rVert\to0$ 
for any $\gamma,\gamma'\in\Gamma$ such that $\gamma\neq\gamma'$. 
\item $\lVert e_n-f_n\rVert_2\to0$. 
\end{enumerate}
\end{prop}
\begin{proof}
Using the proposition above, 
one can prove this in a similar fashion to \cite[Theorem 1.2]{S2}. 
We give a proof here for completeness. 
By Proposition \ref{central2}, 
we may assume $\lVert[a,e_n]\rVert\to0$ for every $a\in A$, 
that is, $(e_n)_n$ is a central sequence. 
Let $\Gamma_0=\{\gamma^{-1}\gamma'
\mid\gamma,\gamma'\in\Gamma,\ \gamma\neq\gamma'\}$. 
Let $\ep>0$. 
It suffices to show that there exists a central sequence $(f_n)_n$ in $A$ 
such that $0\leq f_n\leq1$, 
$\lVert f_n\gamma(f_n)\rVert_2<\ep$ for every $\gamma\in\Gamma_0$ 
and $\lVert e_n-f_n\rVert_2\to0$. 
Put 
\[
e'_n=e_n^{1/2}\left(\sum_{\gamma\in\Gamma_0}\gamma(e_n)\right)e_n^{1/2}. 
\]
Then $\lVert e'_n\rVert_2\to0$ as $n\to\infty$, 
because $\lVert e_n^{1/2}\gamma(e_n)\rVert_2\to0$ 
for each $\gamma\in\Gamma_0$. 
Define a continuous function $g$ by 
\[
g(t)=\begin{cases}\ep^{-1}t&0\leq t\leq\ep \\
1&\ep\leq t. \end{cases}
\]
Let 
\[
f_n=e_n^{1/2}(1-g(e'_n))e_n^{1/2}. 
\]
Notice that $(f_n)_n$ is a central sequence and $0\leq f_n\leq e_n\leq1$. 
Then for any $\gamma\in\Gamma_0$ and $n\in\N$ one has 
\begin{align*}
\lVert f_n\gamma(f_n)\rVert^2
&\leq\lVert f_n\gamma(f_n)f_n\rVert \\
&\leq\left\lVert f_n
\left(\sum_{\gamma\in\Gamma_0}\gamma(e_n)\right)f_n\right\rVert \\
&\leq\left\lVert 
e_n^{1/2}(1-g(e'_n))e_n^{1/2}\left(\sum_{\gamma\in\Gamma_0}\gamma(e_n)\right)
e_n^{1/2}(1-g(e'_n))e_n^{1/2}\right\rVert \\
&\leq\left\lVert e_n^{1/2}
\left(\sum_{\gamma\in\Gamma_0}\gamma(e_n)\right)
e_n^{1/2}(1-g(e'_n))e_n^{1/2}\right\rVert \\
&=\lVert e'_n(1-g(e'_n))e_n^{1/2}\rVert \\
&\leq\lVert e'_n(1-g(e'_n))\rVert<\ep. 
\end{align*}
Moreover 
\[
\lVert f_n-e_n\rVert_2\leq\lVert g(e'_n)\rVert_2\to0
\]
as $n\to\infty$. 
The proof is completed. 
\end{proof}

\begin{thm}\label{wRohlintype}
Let $A$ be a $C^*$-algebra in $\mathcal{C}$ 
with finitely many extremal tracial states and 
let $\alpha:G\curvearrowright A$ be an action of $G$ on $A$. 
Assume either of the following. 
\begin{enumerate}
\item $G=\Z$. 
\item $G=\Z^N$ and $T(A)^\alpha=T(A)$. 
\item $G$ is a finite group. 
\end{enumerate}
Then $\alpha$ has the weak Rohlin property if and only if 
$\alpha$ is strongly outer. 
\end{thm}
\begin{proof}
We have already seen that 
the weak Rohlin property implies strong outerness (Remark \ref{wR>so}). 
Let us prove the other implication. 
Let $E$ be the set of extremal tracial states on $A$. 
For $\tau\in E$, by Remark \ref{RemonC0} (8), 
$\pi_\tau(A)''$ is the hyperfinite II$_1$-factor. 
Set $\rho=\bigoplus_{\tau\in E}\pi_\tau$ and 
$M=\rho(A)''=\bigoplus_{\tau\in E}\pi_\tau(A)''$ 
(cf. \cite[3.8.11, 3.8.12, 3.8.13]{Pedersenbook}). 
There exists an action $\bar\alpha:G\curvearrowright M$ 
satisfying $\bar\alpha_g\circ\rho=\rho\circ\alpha_g$. 
We identify $T(A)$ with $T(M)$ via $\rho$. 
Note that, for a bounded sequence $(x_n)_n$ in $M$, 
$x_n$ converges to zero in the strong operator topology 
if and only if $\lVert x_n\rVert_2$ converges to zero. 

(1)
Let $l$ be the minimum positive integer 
such that $\tau\circ\alpha_l=\tau$ for all $\tau\in T(A)$. 
Let $k$ be a natural number satisfying $k\equiv1\pmod{l}$. 
In the same way as \cite[Lemma 3.1]{K95crelle}, 
one can find a sequence of projections $(q_n)_n$ in $M$ such that 
\[
\lVert[x,q_n]\rVert_2\to0,\quad 
\lVert q_n\bar\alpha_i(q_n)\rVert_2\to0\quad\text{and}\quad 
\left\lVert1-\sum_{i=0}^{kl-1}\bar\alpha_i(q_n)\right\rVert_2\to0
\]
for all $x\in M$ and $i=1,2,\dots,kl{-}1$ as $n\to\infty$. 
Define a projection $p_n\in M$ by 
\[
p_n=\sum_{j=0}^{l-1}\bar\alpha_{kj}(q_n). 
\]
Then we obtain 
\[
\lVert[x,p_n]\rVert_2\to0\quad\text{and}\quad 
\lVert p_n\bar\alpha_i(p_n)\rVert_2\to0
\]
for all $x\in M$ and $i=1,2,\dots,k{-}1$ as $n\to\infty$. 
Also, because of $k\equiv1\pmod{l}$ and $\tau\circ\bar\alpha_l=\tau$, 
we have 
\[
k\tau(p_n)=k\sum_{j=0}^{l-1}\tau(\bar\alpha_{kj}(q_n))
=k\sum_{j=0}^{l-1}\tau(\bar\alpha_j(q_n))
=\sum_{j=0}^{kl-1}\tau(\bar\alpha_j(q_n))\to1. 
\]
for all $\tau\in T(M)$ as $n\to\infty$. 
By Kaplansky's density theorem, 
we may replace $p_n$ with $\rho(e_n)$, 
where $e_n$ is in $A$ and $0\leq e_n\leq1$. 
Thanks to Proposition \ref{central3}, we can obtain the desired sequence. 

(2)
The action $\bar\alpha$ preserves each $\pi_\tau(A)''$ and 
the restriction of $\bar\alpha$ to it is outer. 
It follows from \cite{O} that, 
for any finite subset $F\subset\Z^N$ and $\ep>0$, 
there exist an $(F,\ep)$-invariant finite subset $K\subset\Z^N$ and 
a sequence of projections $(p_n)_n$ in $M$ such that 
\[
\lVert[x,p_n]\rVert_2\to0,\quad 
\lVert\bar\alpha_g(p_n)\bar\alpha_h(p_n)\rVert_2\to0\quad\text{and}\quad 
\tau(p_n)\to1/\lvert K\rvert
\]
for all $x\in M$, $g,h\in K$ with $g\neq h$ and $\tau\in T(M)$ 
as $n\to\infty$. 
The rest of the proof is exactly the same as (1). 

(3)
We claim that 
there exists a sequence of projections $(p_n)_n$ in $M$ such that 
\[
\lVert[x,p_n]\rVert_2\to0,\quad 
\lVert p_n\bar\alpha_g(p_n)\rVert_2\to0\quad\text{and}\quad 
\tau(p_n)\to1/\lvert G\rvert
\]
for all $x\in M$, $g\in G\setminus\{e\}$ and $\tau\in T(M)$ 
as $n\to\infty$. 
Once this is done, the proof is completed in the same way as above. 
To this end, without loss of generality, we may assume that 
$\bar\alpha$ acts transitively on the set of extremal tracial states on $M$, 
that is, $\bar\alpha$ acts ergodically on the center of $M$. 
Fix an extremal tracial state $\tau_0$ and 
let $H=\{g\in G\mid\tau_0\circ\bar\alpha_g=\tau_0\}$. 
From the assumption, 
$\bar\alpha|H:H\curvearrowright\pi_{\tau_0}(A)''$ is outer. 
Therefore, by \cite{O}, 
there exists a sequence of projections $(q_n)_n$ in $\pi_{\tau_0}(A)''$ 
such that 
\[
\lVert[x,q_n]\rVert_2\to0,\quad 
\lVert q_n\bar\alpha_h(q_n)\rVert_2\to0\quad\text{and}\quad 
\tau_0(q_n)\to1/\lvert H\rvert
\]
for all $x\in M$ and $h\in H\setminus\{e\}$ as $n\to\infty$. 
Choose $g_1,g_2,\dots,g_k\in G$ so that $G=g_1H\cup\dots\cup g_kH$, 
where $k=\lvert G\rvert/\lvert H\rvert$. 
Since the central sequence algebra of $\pi_{\tau_0}(A)''$  
(in the sense of von Neumann algebras) is again a II$_1$-factor, 
for each $i=1,2,\dots,k$, we can find a sequence of projections 
$(q_{i,n})_n$ in $\pi_{\tau_0}(A)''$ such that 
\[
q_n=\sum_{i=1}^kq_{i,n},\quad 
\lVert[x,q_{i,n}]\rVert_2\to0\quad\text{and}\quad 
\tau_0(q_{i,n})\to1/\lvert G\rvert
\]
for all $x\in M$ as $n\to\infty$. 
Define a projection $p_n\in M$ by 
\[
p_n=\sum_{i=1}^k\bar\alpha_{g_i}(q_{i,n}). 
\]
It is not so hard to see that $(p_n)_n$ is the desired sequence. 
\end{proof}

\begin{rem}\label{wRohlintypeca}
Theorem \ref{wRohlintype} also holds 
for cocycle actions of $\Z^N$ or a finite group. 
We can prove this in the same way as above, 
because (as mentioned in Definition \ref{ca}) 
a cocycle action $(\alpha,u)$ induces a genuine action on $A_\infty$. 
\end{rem}

\section{$\mathcal{Z}$-stability of crossed products}

The main theorem in this section asserts that 
the crossed product $C^*$-algebra $A\rtimes_\alpha G$ absorbs $\mathcal{Z}$ 
if $A$ is in $\mathcal{C}_0$ and 
$\alpha:G\curvearrowright A$ has the weak Rohlin property 
(Theorem \ref{embedZ}). 
We begin with the definition of the property (SI), 
which was introduced in \cite{S3}. 

\begin{df}[{\cite[Definition 3.3]{S3}}]
We say that a $C^*$-algebra $A$ has the property (SI) 
when for any central sequences $(e_n)_n$ and $(f_n)_n$ in $A$ satisfying 
$0\leq e_n\leq1$, $0\leq f_n\leq1$, 
\[
\lim_{n\to\infty}\max_{\tau\in T(A)}\tau(e_n)=0\quad\text{and}\quad 
\inf_{m\in\N}\liminf_{n\to\infty}\min_{\tau\in T(A)}\tau(f_n^m)>0, 
\]
there exists a central sequence $(s_n)_n$ in $A$ such that 
\[
\lim_{n\to\infty}\lVert s_n^*s_n-e_n\rVert=0\quad\text{and}\quad 
\lim_{n\to\infty}\lVert f_ns_n-s_n\rVert=0. 
\]
\end{df}

The definition above is slightly different from the original one, 
but we can easily check that these two definitions are equivalent. 
Indeed, for any fixed $m\in\N$ one has 
\[
\liminf_{n\to\infty}\min_{\tau\in T(A)}\tau(f_n^n)
\leq\liminf_{n\to\infty}\min_{\tau\in T(A)}\tau(f_n^m), 
\]
and so the definition above is stronger than \cite[Definition 3.3]{S3}. 
Conversely, suppose that 
we are given a sequence $(f_n)_n$ of positive contractions in $A$ satisfying 
\[
c=\inf_{m\in\N}\liminf_{n\to\infty}\min_{\tau\in T(A)}\tau(f_n^m)>0. 
\]
For any $m\in\N$ there exists $n_m\in\N$ such that 
$\tau(f_n^m)>2c/3$ holds for all $\tau\in T(A)$ and $n\geq n_m$. 
We may assume $n_m<n_{m+1}$. 
Let $\delta_m=1-(c/3)^{1/m}$ and 
define $h_m\in C([0,1])$ by $h_m(t)=\min\{t+\delta_m,1\}$. 
For each $n\in\N$, put 
\[
\tilde f_n
=\begin{cases}1&n<n_1\\ h_m(f_n)&n_m\leq n<n_{m+1}. \end{cases}
\]
Evidently $\lVert f'_n-f_n\rVert\to0$ as $n\to\infty$ 
because $\delta_m\to0$ as $m\to\infty$. 
Moreover 
\[
\tau(\tilde f_n^n)=\tau(h_m(f_n)^n)
\geq\tau(f_n^m)-(1-\delta_m)^m>2c/3-c/3=c/3
\]
holds for every $\tau\in T(A)$ and $n\in\N$, 
where $m$ is the natural number satisfying $n_m\leq n<n_{m+1}$. 
Therefore the definition above is weaker than \cite[Definition 3.3]{S3}. 

\begin{lem}\label{SI>projSI}
Let $A$ be a unital simple separable $C^*$-algebra 
with tracial rank zero. 
\begin{enumerate}
\item If $(e_n)_n$ is a central sequence in $A$ satisfying 
$0\leq e_n\leq1$ and 
\[
\lim_{n\to\infty}\max_{\tau\in T(A)}\tau(e_n)=0, 
\]
then there exists a central sequence of projections $(q_n)_n$ in $A$ 
such that 
\[
\lim_{n\to\infty}\lVert q_ne_n-e_n\rVert=0\quad\text{and}\quad 
\lim_{n\to\infty}\max_{\tau\in T(A)}\tau(q_n)=0. 
\]
\item If $(f_n)_n$ is a central sequence in $A$ satisfying 
$0\leq f_n\leq1$ and 
\[
\inf_{m\in\N}\liminf_{n\to\infty}\min_{\tau\in T(A)}\tau(f_n^m)>0, 
\]
then there exists a central sequence of projections $(q_n)_n$ in $A$ 
such that 
\[
\lim_{n\to\infty}\lVert f_nq_n-q_n\rVert=0\quad\text{and}\quad 
\liminf_{n\to\infty}\min_{\tau\in T(A)}\tau(q_n)>0. 
\]
\end{enumerate}
\end{lem}
\begin{proof}
(1) 
Take a finite subset $F\subset A$ and $\ep>0$. 
It suffices to show that, for any sufficiently large $n$, 
there exists a projection $q_n\in A$ such that 
\[
\lVert q_ne_n-e_n\rVert\leq\ep,\quad 
\max_{\tau\in T(A)}\tau(q_n)<\ep\quad\text{and}\quad 
\lVert[a,q_n]\rVert<\ep
\]
for every $a\in F$. 
Use tracial rank zero to find a projection $p\in A$ and 
a finite dimensional unital subalgebra $E\subset pAp$ 
such that the following are satisfied. 
\begin{itemize}
\item For any $a\in F$, $\lVert[a,p]\rVert<\ep/5$. 
\item For any $a\in F$, the distance from $pap$ to $E$ is 
less than $\ep/5$. 
\item For all $\tau\in T(A)$, $\tau(1-p)<\ep$. 
\end{itemize}
Since $[pe_np,b]\to0$ as $n\to\infty$ for every $b\in E$ and 
$E$ is finite dimensional, 
there exists $x_n\in pAp\cap E'$ such that $0\leq x_n\leq1$ 
and $pe_np-x_n\to0$. 
In particular, one has $\max\tau(x_n)\to0$. 
We may further assume that 
$\ep$ is not contained in the spectrum of $x_n$, 
because $pAp\cap E'$ has real rank zero. 
Let $\chi_{[\ep,1]}$ be the characteristic function of the interval $[\ep,1]$ 
and let $y_n=\chi_{[\ep,1]}(x_n)\in pAp\cap E'$. 
Then $\lVert y_nx_n-x_n\rVert<\ep$ and $\tau(y_n)<\ep^{-1}\tau(x_n)$. 
Set $q_n=y_n+(1-p)$. 
It is easy to verify $\lVert[a,q_n]\rVert<\ep$ for any $a\in F$ and $n\in\N$. 
Besides, we have 
\[
\max_{\tau\in T(A)}\tau(q_n)=\max_{\tau\in T(A)}\tau(y_n+1-p)
<\ep^{-1}\max_{\tau\in T(A)}\tau(x_n)+\ep\to\ep
\]
as $n\to\infty$ and 
\[
\limsup_{n\to\infty}\lVert q_ne_n-e_n\rVert
=\limsup_{n\to\infty}\lVert(y_n-p)e_n\rVert
=\limsup_{n\to\infty}\lVert y_nx_n-x_n\rVert\leq\ep, 
\]
thereby completing the proof. 

(2) 
Let $d=\inf_m\liminf_n\min_\tau\tau(f_n^m)>0$. 
Take a finite subset $F\subset A$ and $\ep>0$. 
Choose $m\in\N$ so that $(1-\ep)^m<\ep/2$. 
It suffices to show that, for any sufficiently large $n$, 
there exists a projection $q_n\in A$ such that 
\[
\lVert f_nq_n-q_n\rVert\leq\ep,\quad 
\min_{\tau\in T(A)}\tau(q_n)>d-\ep\quad\text{and}\quad 
\lVert[a,q_n]\rVert<\ep
\]
for every $a\in F$. 
Use tracial rank zero to find a projection $p\in A$ and 
a finite dimensional unital subalgebra $E\subset pAp$ 
such that the following are satisfied. 
\begin{itemize}
\item For any $a\in F$, $\lVert[a,p]\rVert<\ep/4$. 
\item For any $a\in F$, the distance from $pap$ to $E$ is 
less than $\ep/4$. 
\item For all $\tau\in T(A)$, $\tau(1-p)<\ep/2$. 
\end{itemize}
Since $[pf_np,b]\to0$ as $n\to\infty$ for every $b\in E$ and 
$E$ is finite dimensional, 
there exists $x_n\in pAp\cap E'$ such that $0\leq x_n\leq1$ 
and $pf_np-x_n\to0$. 
In particular, 
\begin{align*}
\liminf_{n\to\infty}\min_{\tau\in T(A)}\tau(x_n^m)
&=\liminf_{n\to\infty}\min_{\tau\in T(A)}\tau(pf_n^mp) \\
&\geq\liminf_{n\to\infty}\min_{\tau\in T(A)}\tau(f_n^m)-\tau(1-p)
>d-\ep/2. 
\end{align*}
We may further assume that 
$1-\ep$ is not contained in the spectrum of $x_n$, 
because $pAp\cap E'$ has real rank zero. 
Let $\chi_{[1-\ep,1]}$ be the characteristic function 
of the interval $[1{-}\ep,1]$ 
and let $q_n=\chi_{[1-\ep,1]}(x_n)\in pAp\cap E'$. 
Then $\lVert x_nq_n-q_n\rVert<\ep$ and $\tau(q_n)>\tau(x_n^m)-\ep/2$. 
It is easy to verify $\lVert[a,q_n]\rVert<\ep$. 
Besides, we have 
\[
\liminf_{n\to\infty}\min_{\tau\in T(A)}\tau(q_n)
\geq\liminf_{n\to\infty}\min_{\tau\in T(A)}\tau(x_n^m)-\ep/2
>d-\ep
\]
and 
\[
\limsup_{n\to\infty}\lVert f_nq_n-q_n\rVert
=\limsup_{n\to\infty}\lVert pf_npq_n-q_n\rVert
=\limsup_{n\to\infty}\lVert x_nq_n-q_n\rVert
\leq\ep, 
\]
which completes the proof. 
\end{proof}

\begin{lem}\label{C0hasSI}
Any $C^*$-algebra in $\mathcal{C}_0$ has the property (SI). 
\end{lem}
\begin{proof}
By Remark \ref{RemonC0} (1), \cite[Theorem 5.2]{L04Duke} 
and \cite[Proposition 3.5]{S3}, 
it suffices to show that any unital simple AH algebra $A$ 
with real rank zero and slow dimension growth has the property (SI). 
Suppose that 
we are given two central sequences $(e_n)_n$ and $(f_n)_n$ in $A$ satisfying 
$0\leq e_n\leq1$, $0\leq f_n\leq1$, 
\[
\lim_{n\to\infty}\max_{\tau\in T(A)}\tau(e_n)=0\quad\text{and}\quad 
\inf_{m\in\N}\liminf_{n\to\infty}\min_{\tau\in T(A)}\tau(f_n^m)>0. 
\]
By means of the lemma above, 
we can find central sequences of projections $(p_n)_n$ and $(q_n)_n$ in $A$ 
such that 
\[
\lim_{n\to\infty}\lVert p_ne_n-e_n\rVert=0,\quad 
\lim_{n\to\infty}\max_{\tau\in T(A)}\tau(p_n)=0, 
\]
\[
\lim_{n\to\infty}\lVert f_nq_n-q_n\rVert=0\quad\text{and}\quad 
\liminf_{n\to\infty}\min_{\tau\in T(A)}\tau(q_n)>0. 
\]
By \cite[Lemma 3.3]{M09}, 
there exists a central sequence of partial isometries $(v_n)_n$ in $A$ 
satisfying $v_n^*v_n=p_n$ and $v_nv_n^*\leq q_n$. 
Set $s_n=v_ne_n^{1/2}$. 
Then 
\[
\lim_{n\to\infty}\lVert s_n^*s_n-e_n\rVert=0\quad\text{and}\quad 
\lim_{n\to\infty}\lVert f_ns_n-s_n\rVert=0, 
\]
as desired. 
\end{proof}

\begin{lem}[{\cite[Proposition 5.1]{RW},\cite[Proposition 2.1]{S3}}]\label{c&s}
The $C^*$-algebra $I(k,k{+}1)$ is isomorphic to the universal $C^*$-algebra 
generated by $k{+}1$ elements $c_1,c_2,\dots,c_k,s$ satisfying 
\[
c_1\geq0,\quad c_ic_j^*=\begin{cases}c_1^2&i{=}j\\0&i{\neq}j,\end{cases}\quad 
\sum_{i=1}^kc_i^*c_i+s^*s=1,\quad c_1s=s. 
\]
\end{lem}

\begin{lem}\label{s}
Let $c_1,c_2,\dots,c_k,s$ be the generators of $I(k,k{+}1)$ 
described in the lemma above. 
\begin{enumerate}
\item For any $m\in\N$ and $\tau\in T(I(k,k{+}1))$, 
one has $\tau(c_1^m)\leq k^{-1}$. 
\item For any $m\in\N$ and $\ep>0$, 
there exists a unital embedding $\phi:I(k,k{+}1)\to\mathcal{Z}$ such that 
$k^{-1}-\ep<\omega(\phi(c_1^m))<k^{-1}$, 
where $\omega$ denotes the unique tracial states on $\mathcal{Z}$. 
\end{enumerate}
\end{lem}
\begin{proof}
(1) 
For each $i=2,3,\dots,k$, 
\[
\tau((c_i^*c_i)^{m/2})=\tau((c_ic_i^*)^{m/2})=\tau(c_1^m). 
\]
We also have $c_1^m+(c_2^*c_2)^{m/2}+\dots+(c_k^*c_k)^{m/2}\leq1$. 
Hence $\tau(c_1^m)\leq k^{-1}$. 

(2) 
We may assume that the generator $c_1\in I(k,k{+}1)$ is given by 
\[
c_1(t)=u(t)\left(\sum_{i=1}^ke_{1,1}^{(k)}\otimes e_{i,i}^{(k+1)}
+\cos(\pi t/2)e_{1,1}^{(k)}\otimes e_{k+1,k+1}^{(k+1)}\right)u(t)^*, 
\]
where $u\in C([0,1],M_k\otimes M_{k+1})$ is a unitary, and 
$e_{i,j}^{(k)}$ and $e_{i,j}^{(k+1)}$ are systems of matrix units 
for $M_k$ and $M_{k+1}$, respectively (see \cite[Section 2]{S3}). 
Suppose that $m\in\N$ and $\ep>0$ are given. 
Let $\mu$ be a probability measure on $[0,1]$ such that 
\[
1-\ep<\int_0^1(\cos(\pi t/2))^m\,d\mu(t)<1
\]
and the support of $\mu$ is $[0,1]$. 
By \cite[Theorem 2.1 (i)]{R04IJM}, 
there exists a unital embedding $\phi:I(k,k{+}1)\to\mathcal{Z}$ such that 
\[
\omega(\phi(f))=\int_0^1\tr(f(t))\,d\mu(t), 
\]
where $\tr$ is the normalized trace on $M_k\otimes M_{k+1}$. 
It follows that 
\[
\omega(\phi(c_1^m))=\int_0^1\frac{k+(\cos(\pi t/2))^m}{k(k+1)}\,d\mu(t)
\]
is between $k^{-1}-\ep$ and $k^{-1}$. 
\end{proof}

\begin{lem}\label{CuntzPedersen}
Let $c$ be a positive element of 
a unital $C^*$-algebra $A$ with $T(A)$ non-empty and 
let $\theta\in\R$. 
Suppose that $(f_n)_n$ is a central sequence of $A$ 
such that $\lVert f_n\rVert\leq1$. 
Then we have 
\[
\limsup_{n\to\infty}
\max_{\tau\in T(A)}\lvert\tau(cf_n)-\theta\tau(f_n)\rvert
\leq2\max_{\tau\in T(A)}\lvert\tau(c)-\theta\rvert. 
\]
\end{lem}
\begin{proof}
Put $\delta=\max_{\tau\in T(A)}\lvert\tau(c)-\theta\rvert$. 
Take $\ep>0$ arbitrarily. 
We use the notation in \cite{CP}. 
By \cite[Proposition 2.7, 2.8]{CP}, 
there exist positive elements $a,b\in A$ such that 
$a\sim b$ and $\lVert(c-\theta)-(a-b)\rVert<\delta+\ep$. 
Let $d=(c-\theta)-(a-b)$. 
Then $c+d_-+b=\theta+d_++a$. 
It follows from \cite[Lemma 2.1]{CP} that 
there exist $u_1,u_2,\dots,u_m\in A$ satisfying 
\[
\left\lVert c-\sum u_i^*u_i\right\rVert<\delta+\ep,\quad 
\left\lVert\theta-\sum u_iu_i^*\right\rVert<\delta+\ep. 
\]
Since $(f_n)_n$ is central, for sufficiently large $n$, 
one has $\lVert u_i^*[u_i,f_n]\rVert<m^{-1}\ep$ for $i=1,2,\dots,m$. 
Hence we get 
\[
\tau(cf_n)\approx_{\delta+\ep}\sum\tau(u_i^*u_if_n)
\approx_\ep\sum\tau(u_iu_i^*f_n)
\approx_{\delta+\ep}\theta\tau(f_n)
\]
for every $\tau\in T(A)$, thereby completing the proof. 
\end{proof}

\begin{thm}\label{embedZ}
Let $A$ be a $C^*$-algebra in $\mathcal{C}_0$ and 
let $G$ be $\Z^N$ or a finite group. 
Suppose that 
$\alpha:G\curvearrowright A$ is an action with the weak Rohlin property. 
Then there exists a unital embedding of $\mathcal{Z}$ 
into $(A_\infty)^\alpha$. 
\end{thm}
\begin{proof}
First we consider the case $G=\Z^N$. 
As shown in \cite{JS}, 
$\mathcal{Z}$ is an inductive limit of $I(k,k{+}1)$'s. 
Hence, by \cite[Proposition 2.2]{TW08CJM}, 
it suffices to construct 
a unital homomorphism from $I(k,k{+}1)$ to $(A_\infty)^\alpha$. 
Let $c_1,c_2,\dots,c_k,s$ be the generators of $I(k,k{+}1)$ 
described in Lemma \ref{c&s}. 
We say that $x\in A^\infty$ converges uniformly to $\theta\in\C$ on $T(A)$ 
if 
\[
\lim_{n\to\infty}\max_{\tau\in T(A)}\lvert\tau(x_n)-\theta\rvert=0, 
\]
where $(x_n)_n$ is a representative of $x$. 
Since $A$ absorbs $\mathcal{Z}$ tensorially and 
$\mathcal{Z}$ is isomorphic to 
the infinite tensor product of copies of $\mathcal{Z}$, 
we may think of $A$ as $A\otimes\bigotimes_{n\in\N}\mathcal{Z}$. 
By Lemma \ref{s} (2), 
there exists a unital embedding $\phi_n:I(k,k{+}1)\to\mathcal{Z}$ 
such that $k^{-1}-n^{-1}<\omega(\phi_n(c_1^n))<k^{-1}$. 
We regard $\phi_n$ as an embedding into the $n{+}1$-st 
tensor component of $A\otimes\bigotimes_{n\in\N}\mathcal{Z}$. 
Then 
\[
I(k,k{+}1)\ni d\mapsto(\phi_1(d),\phi_2(d),\phi_3(d),\dots)\in\ell(\N,A)
\]
gives rise to an embedding $\phi:I(k,k{+}1)\to A_\infty$ and 
$\phi(c_1^m)$ converges uniformly to $k^{-1}$ on $T(A)$ for any $m\in\N$ 
by Lemma \ref{s} (1). 
We would like to modify $\phi$ 
so that its image is fixed by the action $\alpha$. 

Let $F\subset G$ be a finite generating set of $G$. 
Take $\ep>0$ and $m\in\N$ arbitrarily. 
Choose $\delta>0$ so that 
\[
\delta(1+\lvert F\rvert+\dots+\lvert F\rvert^{m-1})<\ep. 
\]
We can find an $(F,\delta)$-invariant finite subset $K\subset G$ 
and $f\in A_\infty$ such that $0\leq f\leq1$, 
\[
\alpha_g(f)\alpha_h(f)=0
\]
for $g,h\in K$ with $g\neq h$ 
and $f$ converges uniformly to $\lvert K\rvert^{-1}$ on $T(A)$. 
Let $(f_n)_n$ and $(d_{i,n})_n\in\ell^\infty(\N,A)$ be 
representative sequences of $f$ and $\phi(c_i)$, respectively. 
We may assume that $(f_n)_n$ and $(d_{1,n})_n$ are positive contractions. 
We can find a subsequence $(f_{l_n})_n$ of $(f_n)_n$ such that 
$\lVert[f_{l_n},d_{i,n}]\rVert\to0$ for all $i=1,2,\dots,k$ 
because $(f_n)_n$ is a central sequence. 
In addition, by Lemma \ref{CuntzPedersen}, we may further assume that 
\[
\max_{\tau\in T(A)}\lvert\tau(d_{1,n}^mf_{l_n})-k^{-1}\tau(f_{l_n})\rvert
<2\max_{\tau\in T(A)}\lvert\tau(d_{1,n}^m)-k^{-1}\rvert+1/n
\]
holds, where the right hand side converges to zero as $n\to\infty$. 
Consequently, by replacing $(f_n)_n$ with $(f_{l_n})_n$, 
we may assume that 
$f$ commutes with $\phi(c_1),\phi(c_2),\dots,\phi(c_k)$ and that 
$\phi(c_1^m)f$ converges uniformly to $(k\lvert K\rvert)^{-1}$ on $T(A)$. 
(Such a reindexation trick is rather standard and 
is frequently used in this proof. See \cite[5.3]{O}.) 

Define a function $\ell:K\to\N$ by 
\[
\ell(g)=\min\{n\in\N\mid\exists g_1,g_2,\dots,g_n\in F,\ 
g{+}g_1{+}\dots{+}g_n\notin K\}. 
\]
One has $\lvert\ell^{-1}(1)\rvert\leq\delta\lvert K\rvert$, 
because $K$ is $(F,\delta)$-invariant. 
Then 
$\lvert\ell^{-1}(n)\rvert\leq\delta\lvert F\rvert^{n-1}\lvert K\rvert$ 
is obtained inductively, so that 
\[
\lvert\ell^{-1}(\{1,2,\dots,m\})\rvert
\leq\delta(1+\lvert F\rvert+\dots+\lvert F\rvert^{m-1})
\lvert K\rvert<\ep\lvert K\rvert. 
\]
Let $h:K\to[0,1]$ be the function 
defined by $h(g)=m^{-1}\min\{\ell(g){-}1,m\}$. 
The estimate above implies 
$\lvert h^{-1}(1)\rvert>(1{-}\ep)\lvert K\rvert$. 
We define $\tilde c_i\in A_\infty$ by 
\[
\tilde c_i=\sum_{g\in K}h(g)\alpha_g(\phi(c_i)f^{1/m}). 
\]
It is easy to see that 
$\lVert\tilde c_i-\alpha_g(\tilde c_i)\rVert$ is not greater than $1/m$ 
for all $g\in F$. 
Since $f$ commutes with $\phi(c_1),\phi(c_2),\dots,\phi(c_k)$, one can check 
\[
\tilde c_1\geq0,\quad 
\tilde c_i\tilde c_j^*=\begin{cases}\tilde c_1^2&i{=}j\\0&i{\neq}j\end{cases}
\]
and $\sum_{i=1}^k\tilde c_i^*\tilde c_i\leq1$. 
Moreover, 
$\tilde c_1^m$ converges uniformly to 
$\sum_{g\in K}h(g)^m/(k\lvert K\rvert)$ on $T(A)$, 
which is between $k^{-1}(1-\ep)$ and $k^{-1}$. 

Since $m\in\N$ and $\ep>0$ were arbitrary, 
the standard argument on central sequences allows us to assume that 
$\tilde c_i$'s are in $(A_\infty)^\alpha$ and 
that $\tilde c_1^m$ converges uniformly to $1/k$ on $T(A)$ for any $m\in\N$. 
It follows that 
\[
1-\sum_{i=1}^k\tilde c_i^*\tilde c_i\in(A_\infty)^\alpha
\]
converges uniformly to zero on $T(A)$. 
We would like to construct the image of the generator $s$. 
Take $\ep>0$ arbitrarily. 
By using Theorem \ref{wRohlintype} again, 
one can find an $(F,\ep)$-invariant finite subset $L\subset G$ and 
$e\in A_\infty$ such that $0\leq e\leq1$, 
\[
\alpha_g(e)\alpha_h(e)=0
\]
for $g,h\in L$ with $g\neq h$ 
and $e$ converges uniformly to $\lvert L\rvert^{-1}$ on $T(A)$. 
Let $(e_n)_n\in\ell^\infty(\N,A)$ be a representative sequence of $e$ 
such that $0\leq e_n\leq1$. 
One can choose a sequence of natural numbers $(m_n)_n$ 
satisfying $\lim_{n\to\infty}m_n=\infty$ so that 
$(e_n^{1/m_n})_n$ is still a central sequence and 
\[
\lim_{n\to\infty}\alpha_g(e_n^{1/m_n})\alpha_h(e_n^{1/m_n})=0
\]
for all $g,h\in L$ with $g\neq h$. 
Therefore, by replacing $e_n$ with $e_n^{1/m_n}$, 
we may assume that 
$e^m$ converges uniformly to $\lvert L\rvert^{-1}$ on $T(A)$ for any $m\in\N$. 
In the same way as before, 
we may further assume that $e$ commutes with $\tilde c_1$ and that 
$\tilde c_1^me^m$ converges uniformly to $(k\lvert L\rvert)^{-1}$ on $T(A)$ 
for every $m\in\N$. 
By Lemma \ref{C0hasSI}, $A$ has the property (SI). 
Hence, there exists $r\in A_\infty$ such that 
\[
\tilde c_1er=r\quad\text{and}\quad
r^*r=1-\sum_{i=1}^k\tilde c_i^*\tilde c_i. 
\]
Define $\tilde s\in A_\infty$ by 
\[
\tilde s=\frac{1}{\sqrt{\lvert L\rvert}}\sum_{g\in L}\alpha_g(r). 
\]
Clearly 
$\tilde s^*\tilde s=r^*r=1-\sum_{i=1}^k\tilde c_i^*\tilde c_i$. 
By $0\leq\tilde c_1e\leq\tilde c_1\leq1$, 
we also have $\tilde c_1\tilde s=\tilde s$. 
Moreover, for any $h\in F$, 
\begin{align*}
\lVert\tilde s-\alpha_h(\tilde s)\rVert
&\leq\frac{1}{\sqrt{\lvert L\rvert}}
\left\lVert\sum_{g\in L\setminus h^{-1}L}\alpha_g(r)\right\rVert
+\frac{1}{\sqrt{\lvert L\rvert}}
\left\lVert\sum_{g\in L\setminus hL}\alpha_g(r)\right\rVert \\
&\leq\frac{\lvert L\setminus h^{-1}L\rvert^{1/2}}{\lvert L\rvert^{1/2}}
+\frac{\lvert L\setminus hL\rvert^{1/2}}{\lvert L\rvert^{1/2}}
\leq2\sqrt{\ep}. 
\end{align*}
Since $\ep>0$ was arbitrary, 
the standard argument on central sequences allows us to assume that 
$\tilde s$ is in $(A_\infty)^\alpha$. 
Consequently, from Lemma \ref{c&s}, 
we can obtain a unital homomorphism $\psi:I(k,k{+}1)\to(A_\infty)^\alpha$ 
as desired. 

When $G$ is a finite group, the proof is much easier. 
Indeed we can define $\tilde c_i\in A_\infty$ by 
\[
\tilde c_i=\sum_{g\in G}\alpha_g(\phi(c_i)f^{1/m}), 
\]
where $m$ is a sufficiently large natural number. 
The construction of the element $\tilde s$ is the same as above. 
\end{proof}

\begin{thm}\label{Rordam}
Let $\alpha:G\curvearrowright A$ be an action of 
a countable discrete group $G$ on a unital $C^*$-algebra $A$. 
Suppose that 
there exists a unital embedding of $\mathcal{Z}$ into $(A_\infty)^\alpha$. 
Then $(A,\alpha)$ is strongly cocycle conjugate to 
$(A\otimes\mathcal{Z},\alpha\otimes\id)$. 
In particular, $A\rtimes_\alpha G$ absorbs $\mathcal{Z}$ tensorially. 
\end{thm}
\begin{proof}
This is a trivial generalization of \cite[Theorem 7.2.2]{Rtext}. 
\end{proof}

Combining the theorems above, we obtain the following corollaries. 

\begin{cor}\label{absorbZ0}
Let $A$ be a $C^*$-algebra in $\mathcal{C}_0$ and 
let $G$ be $\Z^N$ or a finite group. 
Suppose that 
$\alpha:G\curvearrowright A$ is an action with the weak Rohlin property. 
Then $(A,\alpha)$ is strongly cocycle conjugate to 
$(A\otimes\mathcal{Z},\alpha\otimes\id)$. 
In particular, $A\rtimes_\alpha G$ absorbs $\mathcal{Z}$ tensorially. 
\end{cor}

\begin{cor}\label{absorbZ}
Let $A$ be a $C^*$-algebra in $\mathcal{C}_0$ 
with finitely many extremal tracial states and 
let $\alpha:G\curvearrowright A$ be a strongly outer action. 
Assume one of the three conditions stated in Theorem \ref{wRohlintype}. 
Then $(A,\alpha)$ is strongly cocycle conjugate to 
$(A\otimes\mathcal{Z},\alpha\otimes\id)$. 
In particular, $A\rtimes_\alpha G$ absorbs $\mathcal{Z}$ tensorially. 
\end{cor}

\begin{rem}\label{absorbZca}
Theorem \ref{embedZ} also holds 
for cocycle actions of $\Z^N$ or a finite group. 
We can prove this in exactly the same way as above, 
because a cocycle action $(\alpha,u)$ gives rise to 
a genuine action on $A_\infty$. 
Theorem \ref{Rordam} also holds for cocycle actions $(\alpha,u)$ 
in the following sense (see also \cite[Theorem 6.3]{IM}): 
If there exists a unital embedding of $\mathcal{Z}$ into $(A_\infty)^\alpha$, 
then there exists an isomorphism $\mu:A\otimes\mathcal{Z}\to A$ such that 
$(\alpha,u)$ is equivalent to 
$(\mu\circ(\alpha\otimes\id)\circ\mu^{-1},\mu(u\otimes1))$. 
Hence one obtains the following: 
If a cocycle action $(\alpha,u)$ on $A\in\mathcal{C}_0$ has 
the weak Rohlin property, 
then there exists an isomorphism $\mu:A\otimes\mathcal{Z}\to A$ such that 
$(\alpha,u)$ is equivalent to 
$(\mu\circ(\alpha\otimes\id)\circ\mu^{-1},\mu(u\otimes1))$. 
This, together with Remark \ref{wRohlintypeca}, implies 
that Corollary \ref{absorbZ}  also holds for cocycle actions. 
\end{rem}

\section{Closedness of $\mathcal{C}_0$ under taking crossed products}

In this section we will show that 
the class $\mathcal{C}_0$ is closed under taking crossed products 
by strongly outer actions of finite groups or $\Z$ 
satisfying certain mild assumptions 
(Theorem  \ref{closedbyfnt} and \ref{closedbyZ}). 
We begin with the case of finite group actions. 
Note that the following theorem is shown in \cite[Theorem 5.5]{ELPW} 
when $A$ has a unique trace. 
See Definition 1.2 and Lemma 1.16 of \cite{P} 
for the definition of the tracial Rohlin property for finite group actions. 

\begin{thm}\label{finitegrp}
Let $A$ be a unital simple separable $C^*$-algebra with tracial rank zero 
and let $\alpha:G\curvearrowright A$ be a strongly outer action of 
a finite group $G$. 
Suppose that $A$ has finitely many extremal traces. 
Then $\alpha$ has the tracial Rohlin property. 
\end{thm}
\begin{proof}
One can prove this statement 
by using Kishimoto's technique developed in \cite{K95crelle,K96JFA,K98JOT}. 
We sketch a proof for the reader's convenience. 
Let $\rho$, $M$ and $\bar\alpha$ be 
as in the proof of Theorem \ref{wRohlintype}. 
The proof of Theorem \ref{wRohlintype} (3) tells us that 
there exists a sequence of projections $(p_n)_n$ in $M$ such that 
\[
\lVert[x,p_n]\rVert_2\to0,\quad 
\lVert p_n\bar\alpha_g(p_n)\rVert_2\to0\quad\text{and}\quad 
\tau(p_n)\to1/\lvert G\rvert
\]
for all $x\in M$, $g\in G\setminus\{e\}$ and $\tau\in T(M)$ 
as $n\to\infty$. 
By \cite[Lemma 2.15]{OP2}, 
we may replace the projections $p_n$ with $\rho(e_n)$, 
where the $e_n$'s are projections in $A$. 
From \cite[Proposition 4.1]{M09}, 
we can conclude that $\alpha$ has the tracial Rohlin property. 
\end{proof}

The following theorem is due to N. C. Phillips. 

\begin{thm}[{\cite[Theorem 2.6]{P}}]\label{PhiRokhlin}
Let $A$ be a unital simple separable $C^*$-algebra with tracial rank zero. 
Let $\alpha:G\curvearrowright A$ be an action of a finite group $G$ on $A$ 
which has the tracial Rokhlin property. 
Then $A\rtimes_\alpha G$ has tracial rank zero. 
\end{thm}

Using the theorems above, we can show the following. 

\begin{thm}\label{closedbyfnt}
Let $A\in\mathcal{C}_0$ and let $\alpha:G\curvearrowright A$ be 
a strongly outer action of a finite group $G$. 
Suppose that $A$ has finitely many extremal tracial states. 
Then $A\rtimes_\alpha G$ belongs to $\mathcal{C}$. 
\end{thm}
\begin{proof}
Clearly $A\rtimes_\alpha G$ is unital and separable. 
By \cite[Theorem 3.1]{K81CMP}, $A\rtimes_\alpha G$ is simple. 
By Theorem \ref{absorbZ}, $A\rtimes_\alpha G$ is $\mathcal{Z}$-stable. 
By Theorem \ref{finitegrp} and \ref{PhiRokhlin}, 
$(A\otimes Q)\rtimes_{\alpha\otimes\id}G\cong(A\rtimes_\alpha G)\otimes Q$ 
has tracial rank zero, where $Q$ is the universal UHF algebra. 
Hence $A\rtimes_\alpha G$ belongs to $\mathcal{C}$. 
\end{proof}

We now turn to $\Z$-actions. 
See \cite[Definition 1.1]{K95crelle} or \cite[Definition 4.1]{K98JFA} 
for the definition of the Rohlin property for $\Z$-actions. 

\begin{lem}\label{cyclic}
Let $A$ be a unital simple AH algebra 
with real rank zero and slow dimension growth. 
Suppose that $A$ has finitely many extremal traces. 
Let $\alpha:\Z\curvearrowright A$ be a strongly outer action. 
Suppose further that 
$(A_\infty)^\alpha$ admits a unital embedding of a UHF algebra. 
Then $\alpha$ has the Rohlin property. 
\end{lem}
\begin{proof}
We regard $\alpha$ as a single automorphism. 
By \cite[Theorem 4.7]{M09}, it suffices to show that 
$\alpha$ possesses the tracial cyclic Rohlin property 
in the sense of \cite{LO}. 
Suppose that we are given $m\in\N$ and $\ep>0$. 
The proof uses a technique developed in \cite{K95crelle}. 
Choose $k_1,k_2\in\N$ so that 
\[
\frac{1}{k_1}+\frac{1}{\sqrt{k_1}}<\ep\quad\text{and}\quad 
\frac{k_1+k_2}{2k_1+k_2-1}>\sqrt{1-\ep}. 
\]
Put $l=m(2k_1+k_2-1)$. 

As in the proof of Theorem \ref{embedZ}, 
we say that $x\in A^\infty$ converges uniformly to $\theta\in\C$ on $T(A)$ 
if $\max_{\tau\in T(A)}\lvert\tau(x_n)-\theta\rvert\to0$ as $n\to0$, 
where $(x_n)_n$ is a representative of $x$. 
By (the proof of) \cite[Theorem 4.2]{M09} and Remark \ref{so=uo}, 
for any $l\in\N$, 
one can find a projection $e\in A_\infty$ 
converging uniformly to $1/l$ on $T(A)$ 
and satisfying $e\alpha^i(e)=0$ for all $i=1,2,\dots,l{-}1$. 
Choose $\theta\in\R$ 
so that $l^{-1}(1-\ep)^{1/2}<\theta<l^{-1}(1-\ep)^{1/3}$. 
Since there exists a unital homomorphism 
from a UHF algebra to $(A_\infty)^\alpha$, 
we can find a projection $f\in(A_\infty)^\alpha$ 
converging uniformly to $\theta$ on $T(A)$. 
It follows from \cite[Lemma 3.7]{M09} that 
there exists a partial isometry $v\in A_\infty$ 
such that $v^*v=f$ and $vv^*\leq e$. 
Put $e'=vfv^*$ and $w=\alpha(v)v^*$. 
Then $e'$ is a projection in $A_\infty$ 
which converges uniformly to $\theta$ on $T(A)$ 
and satisfies $e'\alpha^i(e')=0$ for all $i=1,2,\dots,l{-}1$. 
In addition, $w$ is a partial isometry in $A_\infty$ satisfying 
$w^*w=e'$ and $ww^*=\alpha(e')$. 
Then we can mimic the argument in \cite[Lemma 4.3]{K95crelle} and 
obtain a projection $p\in A_\infty$ satisfying 
\[
p\leq\sum_{j=0}^{2k_1+k_2-2}\alpha^{jm}(e'),\quad 
p\alpha^i(p)=0\quad\text{for any }i=1,2,\dots,m{-}1
\]
and 
\[
\lVert p-\alpha^m(p)\rVert<\frac{1}{k_1}+\frac{1}{\sqrt{k_1}}<\ep. 
\]
Besides $p$ converges uniformly to $\theta(k_1+k_2)$ on $T(A)$ and 
\[
\theta(k_1+k_2)>l^{-1}(1-\ep)^{1/2}(k_1+k_2)>m^{-1}(1-\ep). 
\]
Since $\ep>0$ is arbitrary, 
we can conclude that $\alpha$ has the tracial cyclic Rohlin property. 
See also \cite[Theorem 3.4]{LO} or \cite[Theorem 4.4]{M09}. 
\end{proof}

We would like to recall A. Kishimoto's results 
for automorphisms of unital simple AT algebras with real rank zero. 
For a unital $C^*$-algebra $A$, 
$\aInn(A)$ stands for the group of approximately inner automorphisms of $A$ 
and $\HInn(A)$ stands for the subgroup of $\aInn(A)$ 
consisting of automorphisms homotopic to an inner automorphism. 
Evidently $\HInn(A)$ is a normal subgroup of $\aInn(A)$. 

\begin{thm}[{\cite[Corollary 2.3]{K98JFA}}]\label{Kishimoto1}
Let $A$ be a unital simple AT algebra with real rank zero. 
Suppose that $A$ has a unique trace and $K_1(A)\neq\Z$. 
Then $\aInn(A)/\HInn(A)$ is isomorphic to 
\[
\Ext(K_1(A),K_0(A))\oplus\Ext(K_0(A),K_1(A)). 
\]
\end{thm}

See \cite[Corollary 3.5]{K99IJM} and \cite[Theorem 6.3]{M01} 
for further developments of the theorem above. 

\begin{thm}[{\cite[Corollary 6.7]{K98JFA}}]\label{Kishimoto2}
Let $A$ be a unital simple AT algebra with real rank zero 
and with a unique trace. 
Let $B$ be a UHF algebra and let $\alpha\in\HInn(A)$. 
Then $\alpha$ has the Rohlin property if and only if 
$(A\rtimes_\alpha\Z)\otimes B$ is 
a unital simple AT algebra with real rank zero and with a unique trace. 
\end{thm}

See \cite[Theorem 6.4]{K98JFA} and \cite[Remark 3.7]{K99IJM} 
for relevant results. 

We need one more lemma. 

\begin{lem}\label{power}
Let $A$ be a unital nuclear $C^*$-algebra with a unique trace and 
let $\alpha\in\Aut(A)$. 
Let $k\in\N$ and let $C=A\rtimes_{\alpha^k}\Z$. 
Then there exists an action 
$\gamma:\Z/k\Z\curvearrowright C\otimes M_k$ such that 
\[
(C\otimes M_k)\rtimes_\gamma\Z/k\Z\cong(A\rtimes_\alpha\Z)\otimes M_k. 
\]
Moreover, if the $\Z$-action generated by $\alpha$ is strongly outer, 
then $\gamma$ is also strongly outer. 
\end{lem}
\begin{proof}
Let $\beta:\Z/k\Z\curvearrowright A\rtimes_\alpha\Z$ be the restriction of 
the dual action $\hat\alpha:\T\curvearrowright A\rtimes_\alpha\Z$ 
to $\Z/k\Z\subset\T$. 
Then it is easy to see 
$(A\rtimes_\alpha\Z)\rtimes_\beta\Z/k\Z\cong C\otimes M_k$. 
Indeed $(A\rtimes_\alpha\Z)\rtimes_\beta\Z/k\Z$ is the universal $C^*$-algebra 
generated by $A$ and two unitaries $u$, $v$ satisfying 
\[
uau^*=\alpha(a),\quad vav^*=a\quad\forall a\in A,\quad 
vuv^*=\exp(2\pi\sqrt{-1}/k)u,\quad v^k=1. 
\]
The $C^*$-subalgebra $B$ generated by $A$ and $v$ is clearly 
isomorphic to $A\otimes C^*(\Z/k\Z)\cong A\otimes\C^n$. 
Let $p\in B$ be a minimal central projection 
corresponding to a minimal central projection in $\C^n$. 
Then we have $p+upu^*+\dots+u^{k-1}pu^{1-k}=1$, 
and so $(u^ipu^{-j})_{i,j}$ is a system of matrix units. 
It is easy to see that 
$p((A\rtimes_\alpha\Z)\rtimes_\beta\Z/k\Z)p$ is generated by $Ap$ and $u^kp$. 
Therefore $(A\rtimes_\alpha\Z)\rtimes_\beta\Z/k\Z$ is 
isomorphic to $C\otimes M_k$. 
By the Takesaki-Takai duality, one has 
\[
(A\rtimes_\alpha\Z)\rtimes_\beta\Z/k\Z\rtimes_{\hat\beta}\Z/k\Z
\cong(A\rtimes_\alpha\Z)\otimes M_k. 
\]
Hence there exists $\gamma:\Z/k\Z\curvearrowright C\otimes M_k$ such that 
$(C\otimes M_k)\rtimes_\gamma\Z/k\Z\cong(A\rtimes_\alpha\Z)\otimes M_k$. 

Assume further that 
the $\Z$-action generated by $\alpha$ is strongly outer. 
Let $T(A)=\{\tau\}$. 
By \cite[Lemma 4.3]{K96JFA}, $A\rtimes_\alpha\Z$ has a unique trace. 
Since the $\Z$-action generated by $\alpha^k$ is also strongly outer, 
$(A\rtimes_\alpha\Z)\rtimes_\beta\Z/k\Z\cong C\otimes M_k$ has 
a unique trace, too. 
We denote these tracial states by the same symbol $\tau$. 
Let $\bar\alpha$ be the weak extension of $\alpha$ to $\pi_\tau(A)''$ and 
let $\bar\beta$ be the weak extension of $\beta$ to 
$\pi_\tau(A\rtimes_\alpha\Z)''$. 
The action $\bar\beta$ is canonically identified with 
the restriction of the dual action of $\bar\alpha$ to $\Z/k\Z$. 
It is well-known that 
$\pi_\tau(A\rtimes_\alpha\Z)''$ is the hyperfinite II$_1$-factor and that 
the dual action of $\bar\alpha$ on $\pi_\tau(A\rtimes_\alpha\Z)''$ is outer. 
It follows that $\bar\beta$ is outer (i.e. $\beta$ is strongly outer). 
Then $\pi_\tau((A\rtimes_\alpha\Z)\rtimes_\beta\Z/k\Z)''$ is again 
the hyperfinite II$_1$-factor and 
the dual action of $\bar\beta$ is also outer. 
The weak extension of 
$\hat\beta:\Z/k\Z\curvearrowright(A\rtimes_\alpha\Z)\rtimes_\beta\Z/k\Z$ is 
canonically identified with the dual action of $\bar\beta$ 
on $\pi_\tau((A\rtimes_\alpha\Z)\rtimes_\beta\Z/k\Z)''$. 
Therefore $\hat\beta$ is strongly outer. 
\end{proof}

Combining these statements with the results obtained in the previous section, 
we can prove the following theorem. 

\begin{thm}\label{closedbyZ}
Let $A\in\mathcal{C}_0$ and let $\alpha:\Z\curvearrowright A$. 
Suppose that $A$ has finitely many extremal tracial states. 
Then the following are equivalent. 
\begin{enumerate}
\item $\alpha$ is strongly outer. 
\item $\alpha$ has the weak Rohlin property. 
\item $T(A)^\alpha$ is isomorphic to $T(A\rtimes_\alpha\Z)$ 
via the inclusion of $A$ into $A\rtimes_\alpha\Z$. 
\item $(A\rtimes_\alpha\Z)\otimes B$ has real rank zero 
for any UHF algebra of infinite type. 
\end{enumerate}
Furthermore, if $A$ has a unique trace and there exists $k\in\N$ such that 
$K_i(\alpha_k)$ induces the identity on $K_i(A)\otimes\Q$ for $i=0,1$, then 
the four conditions above are also equivalent to the following condition. 
\begin{enumerate}\setcounter{enumi}{4}
\item $A\rtimes_\alpha\Z$ belongs to $\mathcal{C}_0$. 
\end{enumerate}
\end{thm}
\begin{proof}
(1)$\Leftrightarrow$(2) is shown in Theorem \ref{wRohlintype}. 
(2)$\Rightarrow$(3) is shown in Remark \ref{wR>so}. 
(3)$\Rightarrow$(1) follows from 
Proposition 2.3 and Remark 2.4 of \cite{K98JOT}. 
(4)$\Rightarrow$(3) follows from \cite[Proposition 2.2]{K98JOT}. 
Indeed (4) implies 
\[
(\tau\otimes\tau_B)\circ(\hat\alpha_t\otimes\id_B)
=\tau\otimes\tau_B
\]
for $\tau\in T(A\rtimes_\alpha\Z)$, $\tau_B\in T(B)$ and $t\in\T$, 
and so $\tau\circ\hat\alpha_t=\tau$. 

In order to prove (1)$\Rightarrow$(4), 
we assume that $\alpha$ is strongly outer. 
Let $B$ be a UHF algebra of infinite type. 
By the definition of $\mathcal{C}_0$ and Remark \ref{RemonC0} (1), 
$A\otimes B$ is a unital simple AH algebra with real rank zero and 
slow dimension growth. 
In addition $A\otimes B$ has finitely many extremal traces. 
By Remark \ref{sotimesid}, 
$\alpha\otimes\id:\Z\curvearrowright A\otimes B$ is strongly outer. 
Then by Remark \ref{so=uo} and \cite[Theorem 4.2]{M09}, 
$\alpha\otimes\id:\Z\curvearrowright A\otimes B$ has 
the tracial Rohlin property. 
It follows from \cite[Theorem 3.4]{L01TAMS}, \cite[Theorem 6.8]{L01PLMS} 
and \cite[Theorem 4.5]{OP1} that 
the crossed product 
$(A\otimes B)\rtimes_{\alpha\otimes\id}\Z\cong(A\rtimes_\alpha\Z)\otimes B$ 
has real rank zero. 

(5)$\Rightarrow$(4) follows trivially 
from the definition of $\mathcal{C}_0$ and Remark \ref{RemonC0} (1). 

The implication from (1) to (5) needs the full assumption. 
We regard $\alpha$ as a single automorphism. 
First, $A\rtimes_\alpha\Z$ is unital, separable and nuclear. 
In addition, it satisfies the UCT. 
Simplicity follows from \cite[Theorem 3.1]{K81CMP}. 
Corollary \ref{absorbZ} implies that 
$A\rtimes_\alpha\Z$ is $\mathcal{Z}$-stable. 
It remains for us to show that 
$(A\rtimes_\alpha\Z)\otimes Q$ has tracial rank zero, 
where $Q$ is the universal UHF algebra. 
We first claim that 
$C=(A\rtimes_{\alpha^k}\Z)\otimes Q$ has tracial rank zero. 
Since $K_i(\alpha^k\otimes\id)=\id$ on $K_i(A\otimes Q)$ for $i=0,1$ and 
$A\otimes Q$ is a unital simple AT algebra with real rank zero, 
by \cite[Corollary 3.2.8]{Rtext}, 
$\alpha^k\otimes\id$ is in $\aInn(A\otimes Q)$. 
Clearly $A\otimes Q$ has a unique trace and $K_1(A\otimes Q)\neq\Z$. 
For $i=0,1$, 
$K_i(A\otimes Q)\cong K_i(A)\otimes\Q$ is divisible, 
and so $\Ext(K_i(A\otimes Q),K_{1-i}(A\otimes Q))$ is zero. 
Hence Theorem \ref{Kishimoto1} tells us that 
$\alpha^k\otimes\id$ belongs to $\HInn(A)$. 
By Lemma \ref{cyclic} (or \cite[Theorem 2.1]{K98JOT}), 
$\alpha^k\otimes\id\in\Aut(A\otimes Q)$ has the Rohlin property. 
It follows from Theorem \ref{Kishimoto2} that 
$(A\otimes Q)\rtimes_{\alpha^k\otimes\id}\Z\cong C$ is 
a unital simple AT algebra of real rank zero with a unique trace. 
By Lemma \ref{power}, there exists a strongly outer action 
$\gamma:\Z/k\Z\curvearrowright C\otimes M_k$ such that 
\[
(C\otimes M_k)\rtimes_\gamma\Z/k\Z
\cong((A\otimes Q)\rtimes_{\alpha\otimes\id}\Z)\otimes M_k. 
\]
By virtue of Theorem \ref{finitegrp} and Theorem \ref{PhiRokhlin}, 
this algebra has tracial rank zero. 
By \cite[Theorem 5.4]{L01PLMS}, we can conclude that 
$(A\otimes Q)\rtimes_{\alpha\otimes\id}\Z\cong(A\rtimes_\alpha\Z)\otimes Q$ 
has tracial rank zero. 
\end{proof}

The following corollary is a generalization of \cite[Theorem 6.4]{K98JFA} 
(compare also \cite[Theorem 4.8]{M09}). 

\begin{cor}\label{AHclosedbyZ}
Let $A$ be a unital simple AH algebra 
with real rank zero and slow dimension growth 
and let $\alpha:\Z\curvearrowright A$. 
Suppose that $A$ has finitely many extremal tracial states. 
Then the following are equivalent. 
\begin{enumerate}
\item $\alpha$ is strongly outer. 
\item $\alpha$ has the weak Rohlin property. 
\item $T(A)^\alpha$ is isomorphic to $T(A\rtimes_\alpha\Z)$ 
via the inclusion of $A$ into $A\rtimes_\alpha\Z$. 
\item $A\rtimes_\alpha\Z$ has real rank zero. 
\end{enumerate}
Furthermore, if $A$ has a unique trace and there exists $k\in\N$ such that 
$K_i(\alpha_k)$ induces the identity on $K_i(A)\otimes\Q$ for $i=0,1$, then 
the four conditions above are also equivalent to the following condition. 
\begin{enumerate}\setcounter{enumi}{4}
\item $A\rtimes_\alpha\Z$ is a unital simple AH algebra 
with real rank zero and slow dimension growth. 
\end{enumerate}
\end{cor}
\begin{proof}
The equivalence among (1), (2) and (3) is the same as Theorem \ref{closedbyZ}. 
(4)$\Rightarrow$(3) and (1)$\Rightarrow$(4) are also shown in the same way. 
(5)$\Rightarrow$(4) is trivial. 
(1)$\Rightarrow$(5) follows from 
Theorem \ref{closedbyZ}, Remark \ref{RemonC0} (4) and 
the classification theorem \cite[Theorem 5.2]{L04Duke}. 
\end{proof}

\section{Homotopy of an almost invariant unitary}

In this section we establish Lemma \ref{3torus}. 
The following lemma is a corollary of \cite[Theorem 1.5]{GL}. 

\begin{lem}\label{GongLin}
For $i=1,2,3$ and $n\in\N$, 
let $u_{i,n}$ be a unitary of $M_{d_n}$. 
Suppose that the following hold. 
\begin{enumerate}
\item For any $i,j=1,2,3$, 
$\lVert[u_{i,n},u_{j,n}]\rVert\to0$ as $n\to\infty$. 
\item For any $i,j=1,2,3$ and $n\in\N$, 
$\tr(\log(u_{i,n}u_{j,n}u_{i,n}^*u_{j,n}^*))=0$. 
\item For any $(j_1,j_2,j_3)\in\Z^3\setminus\{(0,0,0)\}$, 
$\tr(u_{1,n}^{j_1}u_{2,n}^{j_2}u_{3,n}^{j_3})\to0$ as $n\to\infty$. 
\end{enumerate}
Then there exist unitaries $u'_{i,n}\in M_{d_n}$ such that 
\[
\lim_{n\to\infty}\lVert u'_{i,n}-u_{i,n}\rVert=0
\quad\text{and}\quad 
[u'_{i,n},u'_{j,n}]=0
\]
for any $i,j=1,2,3$ and $n\in\N$. 
\end{lem}
\begin{proof}
By (1), the sequences $(u_{i,n})_n$ of unitaries induce 
a unital homomorphism $\phi:C(\T^3)\to\prod M_{d_n}/\bigoplus M_{d_n}$. 
Condition (2) implies that 
$K_0(\phi)$ kills the three Bott elements in $K_0(C(\T^3))$. 
In other words, $K_*(\phi)$ equals the map 
induced by a unital homomorphism which factors through $\C$. 
By (3), for any $f\in C(\T^3)$, 
letting $(x_n)_n\in\prod M_{d_n}$ be a representative of $\phi(f)$ 
one has 
\[
\lim_{n\to\infty}\tr(x_n)=\int_{\T^3}f(t)\,dt, 
\]
where $dt$ is the normalized Haar measure on $\T^3$. 
It follows from Theorem 1.5 of \cite{GL} (and its proof) that 
the homomorphism $\phi$ lifts to a unital homomorphism 
from $C(\T^3)$ to $\prod M_{d_n}$. 
The proof is completed. 
\end{proof}

\begin{rem}
The lemma above fails if one removes the assumption (3). 
Indeed, as shown in \cite{V} and \cite[Theorem 2.3]{D}, 
there exist sequences $(a_{i,n})_n$, $i=1,2,3$, 
of self-adjoint matrices of norm one such that 
$\lVert[a_{i,n},a_{j,n}]\rVert\to0$ as $n\to\infty$ for any $i,j=1,2,3$, 
yet there exist no commuting triples $(a'_{i,n})_n$ of self-adjoint matrices 
satisfying $\lVert a_{i,n}-a'_{i,n}\rVert\to0$ as $n\to0$ for $i=1,2,3$. 
Then the unitaries $u_{i,n}=\exp(\pi\sqrt{-1}a_{i,n}/2)$ 
give a counter example. 
\end{rem}

We let $\xi_1=(1,0)$ and $\xi_2=(0,1)$ denote the generators of $\Z^2$. 

\begin{lem}
Let $B$ be a UHF algebra of infinite type with a unique trace $\tau$ and 
let $\alpha:\Z^2\curvearrowright B$ be a strongly outer action. 
For any finite subset $F\subset B$ and $\ep>0$, 
we can find a finite subset $G\subset B$, $\delta>0$ and $m\in\N$ 
such that the following holds: 
If $w\in U(B)$ satisfies 
\[
\lVert w-\alpha_{\xi_i}(w)\rVert<\delta,\quad 
\tau(\log(w\alpha_{\xi_i}(w^*)))=0,\quad 
\lVert[w,a]\rVert<\delta\quad\text{and}\quad\lvert\tau(w^j)\rvert<\delta
\]
for any $i=1,2$, $a\in G$ and $j\in\Z$ with $0<\lvert j\rvert<m$, then 
there exists a path of unitaries $\tilde w:[0,1]\to U(B)$ such that 
\[
\tilde w(0)=1,\quad \tilde w(1)=w,\quad 
\lVert\tilde w(t)-\alpha_{\xi_i}(\tilde w(t))\rVert<\ep
\quad\text{and}\quad\lVert[\tilde w(t),a]\rVert<\ep
\]
for all $t\in[0,1]$ and $a\in F$. 
In addition, if $F=\emptyset$, then $G=\emptyset$ is possible. 
\end{lem}
\begin{proof}
Let $(d_n)_n$ be a sequence of natural numbers such that 
$d_n<d_{n+1}$ and 
\[
B\cong M_{d_1}\otimes M_{d_2}\otimes M_{d_3}\otimes\dots. 
\]
We regard $B_n=M_{d_1}\otimes\dots\otimes M_{d_n}$ 
as a subalgebra of $B$. 
Let $x_n\in B_n\cap B_{n-1}'$ be a unitary 
whose spectrum is $\{\zeta_n^k\mid k=1,2,\dots,d_n\}$, 
where $\zeta_n=\exp(2\pi\sqrt{-1}/d_n)$. 
Let $y_n\in B_n\cap B_{n-1}'$ be a unitary 
such that $y_nx_ny_n^*=\zeta_nx_n$. 
Set $z_n=x_1x_2\dots x_n\in B_n$ and 
define $\sigma\in\Aut(B)$ by $\sigma=\lim\Ad z_n$. 
Define $\beta:\Z^2\curvearrowright B\otimes B$ 
by $\beta_{(1,0)}=\sigma\otimes\id$ and $\beta_{(0,1)}=\id\otimes\sigma$. 
By Corollary 6.6 and Remark 6.7 of \cite{KM}, 
$\alpha$ and $\beta$ are strongly cocycle conjugate, 
and so they are approximately conjugate 
in the sense of \cite[Definition 7 (1)]{N1}. 
Hence it suffices to show the assertion only for $\beta$. 
The unique trace on $B\otimes B$ is also denoted by $\tau$. 

We first consider the case $F=\emptyset$. 
Suppose that $w\in U(B\otimes B)$ satisfies 
\[
\lVert w-\beta_{\xi_i}(w)\rVert<\delta,\quad 
\tau(\log(w\beta_{\xi_i}(w^*)))=0
\quad\text{and}\quad\lvert\tau(w^j)\rvert<\delta
\]
for any $i=1,2$ and $j\in\Z$ with $0<\lvert j\rvert<m$. 
We may assume that 
$w$ belongs to $B_{n-1}\otimes B_{n-1}$ for some (large) $n\in\N$. 
Let $u_1=z_n\otimes1$ and $u_2=1\otimes z_n$. 
Then one has 
\[
\lVert[w,u_i]\rVert<\delta\quad\text{and}\quad 
\tau(\log(wu_iw^*u_i^*))=0
\]
for $i=1,2$. 
For any $(j_1,j_2,j_3)\in\Z^3$, 
\begin{align*}
\tau(u_1^{j_1}u_2^{j_2}w^{j_3})
&=\tau((y_n\otimes1)u_1^{j_1}u_2^{j_2}w^{j_3}(y_n\otimes1)^*) \\
&=\tau((y_n\otimes1)(z_n\otimes1)^{j_1}(y_n\otimes1)^*u_2^{j_2}w^{j_3})
=\zeta_n^{j_1}\tau(u_1^{j_1}u_2^{j_2}w^{j_3}), 
\end{align*}
and so $\tau(u_1^{j_1}u_2^{j_2}w^{j_3})=0$ if $0<\lvert j_1\rvert<d_n$. 
Similarly if $0<\lvert j_2\rvert<d_n$, then 
$\tau(u_1^{j_1}u_2^{j_2}w^{j_3})=0$, too. 
By the assumption, $\lvert\tau(w^{j_3})\rvert<\delta$ 
for any $j_3\in\Z$ with $0<\lvert j_3\rvert<m$. 
Therefore, if $\delta>0$ is small enough and $m\in\N$ is large enough, then 
the lemma above implies that 
the triple $(u_1,u_2,w)$ can be approximated 
by a commuting triple of unitaries in $B_n\otimes B_n$. 
Consequently 
we can obtain the desired path of unitaries $w:[0,1]\to U(B_n\otimes B_n)$. 

When $F\neq\emptyset$, we can choose $k\in\N$ so that 
$B_k\otimes B_k$ contains $F$ approximately. 
By letting $G$ be 
a sufficiently large finite subset of the unit ball of $B_k\otimes B_k$, 
we may assume that 
the unitary $w$ is in the relative commutant of $B_k\otimes B_k$. 
Then the same argument works for the relative commutant. 
\end{proof}

\begin{lem}\label{3torus}
Let $B$ be a UHF algebra of infinite type with a unique trace $\tau$ and 
let $\alpha:\Z^2\curvearrowright B$ be a strongly outer action. 
For any finite subset $F\subset B$ and $\ep>0$, 
we can find a finite subset $G\subset B$ and $\delta>0$ 
such that the following holds: 
If $w\in U(B)$ satisfies 
\[
\lVert w-\alpha_{\xi_i}(w)\rVert<\delta,\quad 
\tau(\log(w\alpha_{\xi_i}(w^*)))=0
\quad\text{and}\quad \lVert[w,a]\rVert<\delta
\]
for each $i=1,2$ and $a\in G$, then 
there exists a path of unitaries $\tilde w:[0,1]\to U(B)$ such that 
\[
\tilde w(0)=1,\quad \tilde w(1)=w,\quad 
\lVert\tilde w(t)-\alpha_{\xi_i}(\tilde w(t))\rVert<\ep
\quad\text{and}\quad \lVert[\tilde w(t),a]\rVert<\ep
\]
for all $t\in[0,1]$ and $a\in F$. 
In addition, if $F=\emptyset$, then $G=\emptyset$ is possible. 
\end{lem}
\begin{proof}
By Corollary 6.6 and Remark 6.7 of \cite{KM}, 
we may replace $(B,\alpha)$ with $(B\otimes B,\alpha\otimes\id)$. 
Suppose that we are given a finite subset $F\subset B\otimes B$ and $\ep>0$. 
Applying the lemma above to $F$ and $\ep/2$, 
we get $G\subset B\otimes B$, $\delta>0$ and $m\in\N$. 
We would like to show that $G$ and $\delta$ meet the requirement. 

Suppose that $w\in U(B\otimes B)$ satisfies 
\[
\lVert w-(\alpha_{\xi_i}\otimes\id)(w)\rVert<\delta,\quad 
\tau(\log(w(\alpha_{\xi_i}\otimes\id)(w^*)))=0
\quad\text{and}\quad\lVert[w,a]\rVert<\delta
\]
for every $i=1,2$ and $a\in G$. 
We can find $x,y\in U(\C\otimes B)$ such that 
$yxy^*=\zeta x$, $[x,a]\approx0$ and $[y,a]\approx0$ 
for any $a\in G\cup\{w\}$, where $\zeta=\exp(2\pi\sqrt{-1}/n)$ 
for some $n\geq m$. 
Moreover we may assume that 
there exists a path of unitaries $\tilde x:[0,1]\to U(\C\otimes B)$ such that 
$\tilde x(0)=1$, $\tilde x(1)=x$ and 
$\lVert[\tilde x(t),a]\rVert<\ep/2$ for all $t\in[0,1]$ and $a\in F$. 
The same argument as in the proof of the lemma above shows that 
$\tau((wx)^j)\approx0$ for any $j\in\Z$ with $0<\lvert j\rvert<m$. 
Hence we can apply the lemma above to $wx$ and 
obtain a path of unitaries  $\tilde w:[0,1]\to U(B\otimes B)$ such that 
\[
\tilde w(0)=1,\quad \tilde w(1)=wx,\quad 
\lVert\tilde w(t)-(\alpha_{\xi_i}\otimes\id)(\tilde w(t))\rVert<\ep/2
\]
and $\lVert[\tilde w(t),a]\rVert<\ep/2$ for all $t\in[0,1]$ and $a\in F$. 
Then the path $t\mapsto\tilde w(t)\tilde x(t)^*$ does the job. 
\end{proof}

\section{Uniqueness of $\Z^2$-actions on the Jiang-Su algebra}

In this section we will prove that 
all strongly outer actions of $\Z^2$ on the Jiang-Su algebra $\mathcal{Z}$ are 
strongly cocycle conjugate to each other (Theorem \ref{uniqueZ2}). 
We let $\omega$ denote the unique tracial state on $\mathcal{Z}$. 

Let $\alpha:\Z^2\curvearrowright A$ be an action of $\Z^2$ 
on a unital $C^*$-algebra $A$ 
such that $\alpha_g$ is approximately inner for every $g\in\Z^2$. 
We let $\xi_1=(1,0)$ and $\xi_2=(0,1)$ denote the generators of $\Z^2$. 
A pair of unitaries $(u_1,u_2)$ in $A$ is called 
an almost $\alpha$-cocycle (\cite[Section 3]{KM}), when 
\[
\lVert u_1\alpha_{\xi_1}(u_2)-u_2\alpha_{\xi_2}(u_1)\rVert<1. 
\]
We say that an almost $\alpha$-cocycle $(u_1,u_2)$ is admissible 
if $u_1,u_2\in U(A)_0$ and the associated element 
\[
\kappa(u_1,u_2,\alpha_{\xi_1},\alpha_{\xi_2})\in K_0(A)
\]
(see \cite[Section 3]{KM} for the definition) is zero. 

\begin{lem}\label{admissible}
Let $\alpha:\Z^2\curvearrowright\mathcal{Z}$ be an action of $\Z^2$ 
on $\mathcal{Z}$. 
Then any almost $\alpha$-cocycle is admissible. 
\end{lem}
\begin{proof}
This immediately follows from \cite[Lemma 3.1]{KM}, 
because $K_0(\mathcal{Z})=\Z$. 
\end{proof}

\begin{lem}\label{UHF1}
Let $B$ be a UHF algebra with a unique trace $\tau$ and 
let $\alpha:\Z^2\curvearrowright B$ be a strongly outer action. 
For any finite subset $F\subset B$ and $\ep>0$, 
we can find a finite subset $G\subset B$ and $\delta>0$ 
such that the following holds: 
If $(u_1,u_2)$ is an admissible almost $\alpha$-cocycle in $B$ satisfying 
\[
\lVert u_1\alpha_{\xi_1}(u_2)-u_2\alpha_{\xi_2}(u_1)\rVert<\delta
\quad\text{and}\quad 
\lVert[u_i,a]\rVert<\delta
\]
for $i=1,2$ and $a\in G$, then 
there exists a unitary $v\in B$ such that 
\[
\lVert u_i-v\alpha_{\xi_i}(v^*)\rVert<\ep\quad\text{and}\quad 
\lVert[v,a]\rVert<\ep
\]
for $i=1,2$ and $a\in F$. 
In addition, if $F=\emptyset$, then $G=\emptyset$ is possible. 
\end{lem}
\begin{proof}
Note that $\alpha$ is uniformly outer by Remark \ref{so=uo}. 
It follows from \cite[Theorem 3]{N1} that $\alpha$ has the Rohlin property. 
The proof is by contradiction. 
Assume that there exist a finite subset $F\subset B$ and $\ep>0$ 
such that the assertion does not hold. 
We choose an increasing sequence $\{G_n\}_{n=0}^\infty$ of 
finite subsets of $B$ whose union is dense in $B$. 
From the hypothesis, for each $n\in\N$, 
we would find an admissible almost $\alpha$-cocycle $(u_{1,n},u_{2,n})$ in $B$ 
satisfying 
\[
\lVert u_{1,n}\alpha_{\xi_1}(u_{2,n})-u_{2,n}\alpha_{\xi_2}(u_{1,n})\rVert
<1/n
\quad\text{and}\quad 
\lVert[u_{i,n},a]\rVert<n^{-1}
\]
for $i=1,2$ and $a\in G_n$, and such that 
there does not exist a unitary $v\in B$ satisfying 
\[
\lVert u_{i,n}-v\alpha_{\xi_i}(v^*)\rVert<\ep\quad\text{and}\quad 
\lVert[v,a]\rVert<\ep
\]
for $i=1,2$ and $a\in F$. 
Notice that $(u_{i,n})_n$ is a central sequence for $i=1,2$. 
Let $\tilde u_i\in B_\infty$ be the image of $(u_{i,n})_n$. 
Then one has 
$\tilde u_1\alpha_{\xi_1}(\tilde u_2)=\tilde u_2\alpha_{\xi_2}(\tilde u_1)$, 
and so the pair of unitaries $(\tilde u_1,\tilde u_2)$ gives rise to 
an $\alpha$-cocycle in $B_\infty$ (see \cite[Section 2]{KM}), 
which is admissible because every $(u_{1,n},u_{2,n})$ is admissible. 
Thanks to \cite[Theorem 4.8]{KM}, one can find a unitary $v\in B_\infty$ 
such that $\tilde u_i=v\alpha_{\xi_i}(v^*)$ holds for $i=1,2$. 
This is a contradiction. 

In the case that $F$ is empty, 
we can argue in the same way as above 
using \cite[Theorem 4.7]{KM} instead of \cite[Theorem 4.8]{KM}. 
\end{proof}

\begin{lem}\label{UHF2}
Let $B$ be a UHF algebra of infinite type with a unique trace $\tau$ and 
let $\alpha:\Z^2\curvearrowright B$ be a strongly outer action. 
For any finite subset $F\subset B$, $\ep>0$ and 
$r\in\tau(K_0(B))$ with $\lvert r\rvert<1/2$, 
there exists $v\in U(B)$ such that 
\[
\lVert v-\alpha_{\xi_1}(v)\rVert<\lvert e^{2\pi\sqrt{-1}r}-1\rvert+\ep,\quad 
\lVert v-\alpha_{\xi_2}(v)\rVert<\ep
\]
\[
\tau(\log(v\alpha_{\xi_1}(v^*)))=2\pi\sqrt{-1}r,\quad 
\tau(\log(v\alpha_{\xi_2}(v^*)))=0
\]
and $\lVert[v,a]\rVert<\ep$ for every $a\in F$. 
\end{lem}
\begin{proof}
By Corollary 6.6 and Remark 6.7 of \cite{KM}, 
$\alpha$ is strongly cocycle conjugate to a product type action $\gamma$ 
(constructed in \cite[Lemma 5.3]{KM} for example). 
For such $\gamma$, we can prove the assertion 
by the technique used in the proof of \cite[Lemma 6.2]{KM}. 
\end{proof}

\begin{lem}\label{UHF3}
Let $B$ be a UHF algebra of infinite type with a unique trace $\tau$ and 
let $\alpha:\Z^2\curvearrowright B$ be a strongly outer action. 
For any finite subset $F\subset B$ and $\ep>0$, 
we can find a finite subset $G\subset B$ and $\delta>0$ 
such that the following holds: 
If $(u_1,u_2)$ is an admissible almost $\alpha$-cocycle in $B$ satisfying 
\[
\lVert u_1\alpha_{\xi_1}(u_2)-u_2\alpha_{\xi_2}(u_1)\rVert<\delta,\quad 
\lVert[u_i,a]\rVert<\delta\quad\text{and}\quad 
\Delta_\tau(u_i)=0
\]
for $i=1,2$ and $a\in G$, then 
there exists a unitary $v\in B$ such that 
\[
\lVert u_i-v\alpha_{\xi_i}(v^*)\rVert<\ep,\quad 
\lVert[v,a]\rVert<\ep\quad\text{and}\quad 
\tau(\log(v^*u_i\alpha_{\xi_i}(v)))=0
\]
for $i=1,2$ and $a\in F$. 
In addition, if $F=\emptyset$, then $G=\emptyset$ is possible. 
\end{lem}
\begin{proof}
Suppose that we are given $F\subset B$ and $\ep>0$. 
By applying Lemma \ref{UHF1} to $F$ and $\ep/2$, 
we obtain a finite subset $G\subset B$ and $\delta>0$. 
We would like to show that $G$ and $\delta$ meet the requirement. 
For $(u_1,u_2)$ as above, 
there exists $v\in U(B)$ such that 
\[
\lVert u_i-v\alpha_{\xi_i}(v^*)\rVert<\ep/2\quad\text{and}\quad 
\lVert[v,a]\rVert<\ep/2
\]
for $i=1,2$ and $a\in F$. 
Define $h_i\in B$ by 
\[
h_i=\frac{1}{2\pi\sqrt{-1}}\log(v^*u_i\alpha_{\xi_i}(v)). 
\]
From $\Delta_\tau(u_i)=0$, 
one also obtains $\tau(h_i)\in\tau(K_0(B))$. 
By using Lemma \ref{UHF2} twice, 
we can find a unitary $v_0\in B$ such that 
\[
\lVert v_0-\alpha_{\xi_i}(v_0)\rVert<\ep/2,\quad 
\tau(\log(v_0\alpha_{\xi_i}(v_0^*)))=2\pi\sqrt{-1}\tau(h_i)
\]
for each $i=1,2$ and $\lVert[v_0,a]\rVert<\ep/2$ for every $a\in F$. 
Then 
\begin{align*}
\tau(\log(v_0^*v^*u_i\alpha_{\xi_i}(vv_0)))
&=\tau(\log(v^*u_i\alpha_{\xi_i}(vv_0)v_0^*) \\
&=\tau(\log(v^*u_i\alpha_{\xi_i}(v)))
+\tau(\log(\alpha_{\xi_i}(v_0)v_0^*)) \\
&=0
\end{align*}
for each $i=1,2$. 
Replacing $v$ with $vv_0$, we can complete the proof. 
\end{proof}

\begin{lem}\label{prestability}
Let $\alpha:\Z^2\curvearrowright\mathcal{Z}$ be 
a strongly outer action of $\Z^2$ on $\mathcal{Z}$. 
For any finite subset $F\subset\mathcal{Z}$ and $\ep>0$, 
there exist a finite subset $G\subset\mathcal{Z}$ and $\delta>0$ 
such that the following holds: 
If $(u_1,u_2)$ is an almost $\alpha$-cocycle in $\mathcal{Z}$ 
satisfying 
\[
\Delta_\omega(u_1)=\Delta_\omega(u_2)=0,\quad 
\lVert u_1\alpha_{\xi_1}(u_2)-u_2\alpha_{\xi_2}(u_1)\rVert<\delta
\]
and $\lVert[u_i,a]\rVert<\delta$ for any $a\in G$ and $i=1,2$, then 
there exists a unitary $v\in\mathcal{Z}\otimes\mathcal{Z}$ such that 
\[
\lVert u_i\otimes1-v(\alpha_{\xi_i}\otimes\id)(v^*)\rVert<\ep
\quad\text{and}\quad 
\lVert[v,a\otimes1]\rVert<\ep
\]
for any $a\in F$ and $i=1,2$. 
In addition, if $F=\emptyset$, then $G=\emptyset$ is possible. 
\end{lem}
\begin{proof}
According to \cite[Proposition 3.3]{RW}, 
the Jiang-Su algebra $\mathcal{Z}$ contains a unital subalgebra 
isomorphic to 
\[
Z=\{f\in C([0,1],B_0\otimes B_1)\mid 
f(0)\in B_0\otimes\C, \ f(1)\in \C\otimes B_1\}, 
\]
where $B_0$ and $B_1$ are the UHF algebras of type $2^\infty$ and $3^\infty$, 
respectively. 
We identify $B_0$ and $B_1$ with $B_0\otimes\C$ and $\C\otimes B_1$. 
Set $B=B_0\otimes B_1$ and 
denote the unique trace on $\mathcal{Z}\otimes B$ by $\tau$. 
By Remark \ref{sotimesid}, 
$\alpha\otimes\id$ on $\mathcal{Z}\otimes B_j$ and $\mathcal{Z}\otimes B$ is 
strongly outer. 

The proof is by contradiction. 
Suppose that $(u_{1,n})_n$ and $(u_{2,n})_n$ are central sequences 
of unitaries in $\mathcal{Z}$ satisfying 
\[
\Delta_\omega(u_{1,n})=\Delta_\omega(u_{2,n})=0\quad\text{and}\quad 
\lim_{n\to\infty}
\lVert u_{1,n}\alpha_{\xi_1}(u_{2,n})-u_{2,n}\alpha_{\xi_2}(u_{1,n})\rVert=0. 
\]
By Lemma \ref{admissible}, 
$(u_{1,n},u_{2,n})$ is an admissible almost $\alpha$-cocycle 
(if $n$ is sufficiently large). 
Therefore $(u_{1,n}\otimes1,u_{2,n}\otimes1)$ is 
an admissible almost $\alpha\otimes\id$-cocycle in $\mathcal{Z}\otimes B_j$ 
for each $j=0,1$. 
Evidently one has $\Delta_{\tau}(u_{i,n}\otimes1)=0$ for $i=1,2$. 
Applying Lemma \ref{UHF3} to 
$\alpha\otimes\id:\Z^2\curvearrowright\mathcal{Z}\otimes B_j$, 
we obtain a central sequence of unitaries $(v_{j,n})_n$ 
in $\mathcal{Z}\otimes B_j$ such that 
\[
\lim_{n\to\infty}
\lVert u_{i,n}\otimes1-v_{j,n}(\alpha_{\xi_i}\otimes\id)(v_{j,n}^*)\rVert=0
\]
and 
\[
\tau(\log(v_{j,n}^*(u_{i,n}\otimes1)(\alpha_{\xi_i}\otimes\id)(v_{j,n})))=0. 
\]
Put $w_n=v_{0,n}^*v_{1,n}$. 
Then $w_n$ is a central sequence of unitaries in $\mathcal{Z}\otimes B$ and 
$\lVert w_n-(\alpha_{\xi_i}\otimes\id)(w_n)\rVert\to0$ as $n\to\infty$ 
for each $i=1,2$. 
Moreover, 
\begin{align*}
& \tau(\log(w_n(\alpha_{\xi_i}\otimes\id)(w_n^*))) \\
&=\tau(\log(v_{0,n}^*v_{1,n}(\alpha_{\xi_i}\otimes\id)(v_{1,n}^*v_{0,n}))) \\
&=\tau(\log(v_{1,n}(\alpha_{\xi_i}\otimes\id)(v_{1,n}^*v_{0,n})v_{0,n}^*)) \\
&=\tau(\log((u_{i,n}\otimes1)^*v_{1,n}
(\alpha_{\xi_i}\otimes\id)(v_{1,n}^*v_{0,n})v_{0,n}^*(u_{i,n}\otimes1))) \\
&=\tau(\log((u_{i,n}\otimes1)^*v_{1,n}(\alpha_{\xi_i}\otimes\id)(v_{1,n}^*)))
+\tau(\log(\alpha_{\xi_i}(v_{0,n})v_{0,n}^*(u_{i,n}\otimes1))) \\
&=0. 
\end{align*}
By virtue of Lemma \ref{3torus}, 
there exists a central sequence of unitaries $(\tilde w_n)_n$ 
in $C([0,1],\mathcal{Z}\otimes B)$ satisfying 
\[
\tilde w_n(0)=1,\quad \tilde w_n(1)=w_n
\]
and 
\[
\lim_{n\to\infty}\sup_{t\in[0,1]}
\lVert\tilde w_n(t)-(\alpha_{\xi_i}\otimes\id)(\tilde w_n(t))\rVert=0. 
\]
We define a central sequence of unitaries $(v_n)_n$ 
in $C([0,1],\mathcal{Z}\otimes B)=\mathcal{Z}\otimes C([0,1],B)$ 
by $v_n(t)=v_{0,n}\tilde w_n(t)$. 
It is easy to see that $v_n$ is in $\mathcal{Z}\otimes Z$ and 
\[
\lim_{n\to\infty}
\lVert u_{i,n}\otimes1-v_n(\alpha_{\xi_i}\otimes\id)(v_n^*)\rVert=0
\]
for each $i=1,2$. 
The proof is completed, 
because $\mathcal{Z}\otimes Z$ is 
a subalgebra of $\mathcal{Z}\otimes\mathcal{Z}$. 
\end{proof}

\begin{thm}\label{stability}
Let $\alpha:\Z^2\curvearrowright\mathcal{Z}$ be 
a strongly outer action of $\Z^2$ on $\mathcal{Z}$. 
For any finite subset $F\subset\mathcal{Z}$ and $\ep>0$, 
there exist a finite subset $G\subset\mathcal{Z}$ and $\delta>0$ 
such that the following holds: 
If $(u_1,u_2)$ is an almost $\alpha$-cocycle in $\mathcal{Z}$ 
satisfying 
\[
\Delta_\omega(u_1)=\Delta_\omega(u_2)=0,\quad 
\lVert u_1\alpha_{\xi_1}(u_2)-u_2\alpha_{\xi_2}(u_1)\rVert<\delta
\]
and $\lVert[u_i,a]\rVert<\delta$ for any $a\in G$ and $i=1,2$, then 
there exists a unitary $v\in\mathcal{Z}$ such that 
\[
\lVert u_i-v\alpha_{\xi_i}(v^*)\rVert<\ep\quad\text{and}\quad 
\lVert[v,a]\rVert<\ep
\]
for any $a\in F$ and $i=1,2$. 
In addition, if $F=\emptyset$, then $G=\emptyset$ is possible. 
\end{thm}
\begin{proof}
By Corollary \ref{absorbZ}, 
$(\mathcal{Z},\alpha)$ is strongly cocycle conjugate to 
$(\mathcal{Z}\otimes\mathcal{Z},\alpha\otimes\id)$. 
As $\mathcal{Z}$ is isomorphic to 
the infinite tensor product of $\mathcal{Z}$, 
the assertion follows from the lemma above. 
\end{proof}

\begin{lem}\label{almstexists}
Let $\alpha,\beta:\Z^2\curvearrowright\mathcal{Z}$ be 
two strongly outer actions of $\Z^2$. 
Then there exist $u_1,u_2\in U(\mathcal{Z}^\infty)$ 
such that $u_1\alpha_{\xi_1}(u_2)=u_2\alpha_{\xi_2}(u_1)$ and 
\[
\beta_{\xi_i}(a)=(\Ad u_i\circ\alpha_{\xi_i})(a)
\]
for any $a\in\mathcal{Z}$ and $i=1,2$. 
\end{lem}
\begin{proof}
By \cite[Theorem 1.2]{S2} and \cite[Theorem 1.3]{S3}, 
there exist $\mu\in\Aut(\mathcal{Z})$ and $v_1\in U(\mathcal{Z})$ 
satisfying $\mu\circ\beta_{\xi_1}\circ\mu^{-1}=\Ad v_1\circ\alpha_{\xi_1}$ 
on $\mathcal{Z}$. 
Since any automorphism of $\mathcal{Z}$ is approximately inner, 
we can find $x\in U(\mathcal{Z}^\infty)$ such that 
\[
(\mu\circ\beta_{\xi_2}\circ\mu^{-1})(a)
=(\Ad x\circ\alpha_{\xi_2})(a)
\]
for all $a\in\mathcal{Z}$. 
It is easy to see that 
$y=v_1\alpha_{\xi_1}(x)(x\alpha_{\xi_2}(v_1))^*$ is a unitary 
in $\mathcal{Z}_\infty$. 
From the definition, 
$y$ has a representative sequence $(y_n)_n\in\ell^\infty(\N,\mathcal{Z})$ 
satisfying $\Delta_\omega(y_n)=0$. 
It follows from \cite[Theorem 5.3]{S3} that 
there exists $z\in U(\mathcal{Z}_\infty)$ satisfying 
\[
y^*=z(\Ad v_1\circ\alpha_{\xi_1})(z^*). 
\]
Set $v_2=z^*x\in\mathcal{Z}^\infty$. 
One obtains $v_1\alpha_{\xi_1}(v_2)=v_2\alpha_{\xi_2}(v_1)$ and 
\[
(\mu\circ\beta_{\xi_2}\circ\mu^{-1})(a)=(\Ad v_2\circ\alpha_{\xi_2})(a)
\]
for any $a\in\mathcal{Z}$. 
As $\mu\in\Aut(\mathcal{Z})$ is approximately inner, 
there exists $w\in U(\mathcal{Z}^\infty)$ such that 
$\mu(a)=\Ad w(a)$ for every $a\in\mathcal{Z}$. 
Put $u_i=w^*v_i\alpha_{\xi_i}(w)$ for $i=1,2$. 
Then $u_1$ and $u_2$ meet the requirement. 
\end{proof}

\begin{thm}\label{ccexists}
Let $\alpha,\beta:\Z^2\curvearrowright\mathcal{Z}$ be 
two strongly outer actions of $\Z^2$. 
For any finite subset $F\subset\mathcal{Z}$ and $\ep>0$, 
there exist $u_1,u_2\in U(\mathcal{Z})$ 
such that $u_1\alpha_{\xi_1}(u_2)=u_2\alpha_{\xi_2}(u_1)$ and 
\[
\lVert\beta_{\xi_i}(a)-(\Ad u_i\circ\alpha_{\xi_i})(a)\rVert<\ep
\]
for any $a\in F$ and $i=1,2$. 
\end{thm}
\begin{proof}
Lemma \ref{almstexists} yields 
an almost $\alpha$-cocycle $(u_1,u_2)$ in $\mathcal{Z}$ such that 
\[
\lVert u_1\alpha_{\xi_1}(u_2)-u_2\alpha_{\xi_2}(u_1)\rVert\approx0
\quad\text{and}\quad 
\lVert\beta_{\xi_i}(a)-(\Ad u_i\circ\alpha_{\xi_i})(a)\rVert<\ep/2
\]
for any $a\in F$ and $i=1,2$. 
Clearly we may further assume $\Delta_\omega(u_1)=\Delta_\omega(u_2)=0$. 
From Theorem \ref{stability}, we get $v\in U(\mathcal{Z})$ satisfying 
\[
\lVert u_i-v\alpha_{\xi_i}(v^*)\rVert<\ep/2
\]
for $i=1,2$. 
Therefore 
the unitaries $v\alpha_{\xi_1}(v^*)$ and $v\alpha_{\xi_2}(v^*)$ do the job. 
\end{proof}

\begin{thm}\label{uniqueZ2}
Any two strongly outer actions of $\Z^2$ on the Jiang-Su algebra $\mathcal{Z}$ 
are strongly cocycle conjugate to each other. 
\end{thm}
\begin{proof}
The proof is carried out by the Evans-Kishimoto intertwining argument. 
In each step of this argument, 
we have to use Theorem \ref{ccexists} and 
Theorem \ref{stability} repeatedly. 
We omit the detail, 
because it is exactly the same as \cite[Theorem 5.2]{M08}. 
\end{proof}

\section{Cocycle actions of $\Z^2$}

In this section 
we determine when a strongly outer cocycle action of $\Z^2$ 
on a UHF algebra or the Jiang-Su algebra is equivalent to a genuine action. 

Let $\alpha$ be an automorphism of a unital $C^*$-algebra $A$. 
For a unitary $u\in U(A)_0$ satisfying $\lVert u-\alpha(u)\rVert<2$, 
an element $\kappa(u,\alpha)$ in $\Coker(\id-K_0(\alpha))$ is 
introduced in \cite[Section 7]{IM}. 
The element $\kappa(u,\alpha)$ is obtained 
from the $K_1$-class of the path $v:[0,1]\to U(A)$ such that 
$v(0)=v(1)=1$ and 
\[
\lVert x(t)\alpha(x(t))^*-v(t)\rVert<2, 
\]
where $x:[0,1]\to U(A)$ is a path such that $x(0)=1$ and $x(1)=u$. 
In what follows, 
we consider only the case that $\alpha$ is approximately inner. 
Thus $\kappa(u,\alpha)$ is in $K_0(A)$. 
When $T(A)$ is not empty, for $\tau\in T(A)$, it is easy to see that 
\[
\tau(\kappa(u,\alpha))=\frac{1}{2\pi\sqrt{-1}}\tau(\log(u\alpha(u^*))). 
\]
In particular, $\lvert\tau(\kappa(u,\alpha))\rvert$ is less than $1/2$. 
Hence, if $A$ is $\mathcal{Z}$, then $\kappa(u,\alpha)$ is always zero. 

\begin{lem}
For any $\ep>0$ there exists $\delta>0$ such that the following holds. 
Let $A$ be a unital simple AF algebra with finitely many extremal traces 
and let $\alpha$ be an approximately inner automorphism of $A$ such that 
the $\Z$-action generated by $\alpha$ is strongly outer. 
If $u\in U(A)$ satisfies 
\[
\lVert u-\alpha(u)\rVert<\delta\quad\text{and}\quad \kappa(u,\alpha)=0, 
\]
then there exists a path of unitaries $w:[0,1]\to U(A)$ such that 
\[
w(0)=1,\quad w(1)=u,\quad 
\lVert w(t)-\alpha(w(t))\rVert<\ep
\]
for all $t\in[0,1]$ and $\Lip(w)<4$. 
\end{lem}
\begin{proof}
Clearly $A$ is of infinite dimensional. 
Applying \cite[Lemma 4.2]{KM} to $\ep$, 
we obtain a positive real number $\delta>0$. 
We may assume $\delta<2$. 

Let $x_n$ be a unitary of $M_n(\C)$ 
whose spectrum is $\{\zeta^k\mid k=0,1,\dots,n{-}1\}$, 
where $\zeta=\exp(2\pi\sqrt{-1}/n)$. 
One can find an increasing sequence $\{A_n\}_{n=1}^\infty$ of 
unital finite dimensional subalgebras of $A$ such that 
$\bigcup_n A_n$ is dense in $A$ and 
there exists a unital embedding $\pi_n:M_n\oplus M_{n+1}\to A_n\cap A_{n-1}'$. 
Let $y_n=\pi_n(x_n\oplus x_{n+1})$ and $z_n=y_1y_2\dots y_n$. 
Define an automorphism $\sigma$ of $A$ 
by $\sigma=\lim_{n\to\infty}\Ad z_n$. 
Then $\sigma$ is an approximately inner automorphism with the Rohlin property. 
In particular, 
the $\Z$-action generated by $\sigma$ is strongly outer. 

We would like to show that the assertion holds for $\sigma$. 
Suppose that $u\in U(A)$ satisfies 
$\lVert u-\sigma(u)\rVert<\delta$ and $\kappa(u,\sigma)=0$. 
Without loss of generality, 
we may assume that there exists $n\in\N$ such that $u\in A_n$. 
From the assumption, there exists $m\geq n$ such that 
$\kappa(u,\Ad z_m)$ is zero in $K_0(A_m)$. 
Therefore, by \cite[Lemma 4.2]{KM}, 
one can find a path of unitaries $w:[0,1]\to U(A_m)$ such that 
\[
w(0)=1,\quad w(1)=u,\quad 
\lVert[z_m,w(t)]\rVert<\ep
\]
for all $t\in[0,1]$ and $\Lip(w)<4$. 
Then $w$ is a desired path because $\sigma(w(t))=z_mw(t)z_m^*$. 

Suppose that 
$\alpha$ is an approximately inner automorphism of $A$ such that 
the $\Z$-action generated by $\alpha$ is strongly outer. 
It is well-known that 
$\alpha$ and $\sigma$ are strongly cocycle conjugate 
(see \cite{K95crelle,K98JFA,M09}). 
Hence the assertion also holds for $\alpha$. 
\end{proof}

The following is an adaptation of \cite[Lemma 7.9]{IM}. 

\begin{lem}\label{ocneanu1}
For any $\ep>0$, there exists $\delta>0$ such that the following holds. 
Let $A$ be a unital simple AF algebra with a unique trace and 
let $\alpha,\beta$ be approximately inner automorphisms of $A$ such that 
$\alpha^m\circ\beta^n$ is not weakly inner 
for all $(m,n)\in\Z^2\setminus\{(0,0)\}$. 
Let $u$ and $w$ be unitaries in $A$ such that 
\[
\beta\circ\alpha=\Ad w\circ\alpha\circ\beta,\quad 
\lVert w-1\rVert<\delta,\quad 
\lVert u-\alpha(u)\rVert<\delta\quad\text{and}\quad 
\kappa(u,\alpha)=0. 
\]
Then there exists a unitary $v\in A$ such that 
\[
\lVert v-\alpha(v)\rVert<\ep\quad\text{and}\quad 
\lVert u-v\beta(v)^*\rVert<\ep. 
\]
\end{lem}
\begin{proof}
Choose $m\in\N$ so that $4/m<\ep$. 
Applying the lemma above to $\ep/2$, we obtain $\delta'>0$. 
Choose $\delta>0$ so that 
$(3m{+}1)\delta<\delta'$ and $(5m{-}2)\delta<\ep/2$. 

By \cite[Theorem 5.5]{M09}, 
there exist projections $e$ and $f$ in $A_\infty$ such that 
\[
\alpha(e)=e,\quad \alpha(f)=f,\quad \beta^m(e)=e,\quad \beta^{m+1}(f)=f
\]
and 
\[
\sum_{i=0}^{m-1}\beta^i(e)+\sum_{j=0}^m\beta^j(f)=1. 
\]
Define $u_k\in U(A)$ for $k=0,1,\dots$ 
by $u_0=1$ and $u_{k+1}=u\beta(u_k)$. 
By an elementary estimate, we obtain 
\[
\lVert u_k-\alpha(u_k)\rVert<(3k-2)\delta
\]
for any $k\in\N$. 
Moreover, from \cite[Lemma 7.5]{IM}, 
we can see $\kappa(u_k,\alpha)=0$. 
By applying the lemma above to $u_m$ and $u_{m+1}$, 
we obtain a path of unitaries $y$ and $z$ such that the following hold. 
\begin{itemize}
\item $y(1)=z(1)=1$, $y(0)=u_m$ and $z(0)=u_{m+1}$. 
\item $\lVert y(t)-\alpha(y(t))\rVert<\ep/2$ 
and $\lVert z(t)-\alpha(z(t))\rVert<\ep/2$ for all $t\in[0,1]$. 
\item Both $\Lip(y)$ and $\Lip(z)$ are less than $4$. 
\end{itemize}
Define a unitary $v\in A^\infty$ by 
\[
v=\sum_{k=0}^{m-1}u_k\beta^k(y(k/m))\beta^k(e)
+\sum_{k=0}^mu_k\beta^k(z(k/(m{+}1)))\beta^k(f). 
\]
We can easily see $\lVert u-v\beta(v)^*\rVert<4/m<\ep$. 
Furthermore, 
\[
\lVert v-\alpha(v)\rVert
<(3m-2)\delta+2m\delta+\ep/2<\ep, 
\]
which completes the proof. 
\end{proof}

One can show the following lemma 
in a similar fashion to Lemma \ref{UHF2}. 

\begin{lem}\label{UHF2'}
Let $A$ be a UHF algebra with a unique trace $\tau$ and 
let $\alpha$ be an automorphism of $A$ such that 
the $\Z$-action generated by $\alpha$ is strongly outer. 
Then for any $\ep>0$ and $r\in\tau(K_0(A))$ with $\lvert r\rvert<1/2$, 
there exists $v\in U(A)$ such that 
\[
\lVert v-\alpha(v)\rVert<\lvert e^{2\pi\sqrt{-1}r}-1\rvert+\ep,\quad 
\tau(\log(v\alpha(v^*)))=2\pi\sqrt{-1}r. 
\]
\end{lem}

\begin{lem}\label{ocneanu2}
Let $A$ be a UHF algebra with a unique trace $\tau$ and 
let $\alpha$ be an automorphism of $A$ such that 
the $\Z$-action generated by $\alpha$ is strongly outer. 
Suppose that 
$w\in U(A)$ satisfies $\lVert w-1\rVert<2$ and $\tau(\log w)=0$. 
Then for any $\ep>0$ there exists $u\in U(A)$ such that 
\[
\lVert uw\alpha(u^*)-1\rVert<\ep,\quad 
\lVert u-\alpha(u)\rVert<2\quad\text{and}\quad 
\tau(\log(u\alpha(u^*))=0. 
\]
\end{lem}
\begin{proof}
We may assume $\ep$ is less than $2-\lVert w{-}1\rVert$. 
The automorphism $\alpha$ has the stability, and so 
there exists $u\in U(A)$ such that $\lVert uw\alpha(u^*)-1\rVert<\ep/2$. 
Then the rest of the proof is similar to Lemma \ref{UHF3}. 
One can make use of Lemma \ref{UHF2'} instead of Lemma \ref{UHF2}. 
\end{proof}

\begin{prop}\label{ocneanu3}
For any $\ep>0$, there exists $\delta>0$ such that the following holds. 
Let $A$ be a UHF algebra with a unique trace $\tau$. 
Suppose that 
$\alpha,\beta\in\Aut(A)$ and $w\in U(A)$ satisfy 
\[
\beta\circ\alpha=\Ad w\circ\alpha\circ\beta,\quad 
\lVert w-1\rVert<\delta\quad\text{and}\quad \tau(\log w)=0. 
\]
Suppose further that 
$\alpha^m\circ\beta^n$ is not weakly inner 
for all $(m,n)\in\Z^2\setminus\{(0,0)\}$. 
Then there exist $a,b\in U(A)$ such that 
\[
\lVert a-1\rVert<\ep,\quad \lVert b-1\rVert<\ep
\]
and 
\[
b\beta(a)w\alpha(b)^*a^*=1. 
\]
In particular, 
\[
(\Ad a\circ\alpha)\circ(\Ad b\circ\beta)
=(\Ad b\circ\beta)\circ(\Ad a\circ\alpha). 
\]
\end{prop}
\begin{proof}
Using Lemma \ref{ocneanu1} and Lemma \ref{ocneanu2}, 
one can prove this statement 
in the same way as \cite[Proposition 7.10]{IM}. 
Notice that $\kappa(u,\alpha)$ equals $\tau(\log(u\alpha(u^*)))$, 
because $A$ is a UHF algebra. 
\end{proof}

As in the previous section, 
we let $\xi_1=(1,0)$ and $\xi_2=(0,1)$ be the generators of $\Z^2$. 

\begin{thm}\label{ocneanu4}
Let $(\alpha,u)$ be a strongly outer cocycle action of $\Z^2$ 
on a UHF algebra $A$ with a unique trace $\tau$. 
If $\Delta_\tau(u(\xi_1,\xi_2)u(\xi_2,\xi_1)^*)=0$, then 
$(\alpha,u)$ is equivalent to a genuine action. 
\end{thm}
\begin{proof}
Let $w=u(\xi_2,\xi_1)u(\xi_1,\xi_2)^*$. 
Then 
\[
\alpha_{\xi_2}\circ\alpha_{\xi_1}
=\Ad w\circ\alpha_{\xi_1}\circ\alpha_{\xi_2}
\]
and $\Delta_\tau(w)=0$. 
In a similar fashion to Lemma \ref{UHF3}, for any $\delta>0$, 
one can find $v\in U(A)$ such that 
$\lVert\alpha_{\xi_2}(v^*)wv-1\rVert<\delta$ and 
$\tau(\log(\alpha_{\xi_2}(v^*)wv))=0$. 
Letting $w'=\alpha_{\xi_2}(v)^*wv$, we have 
\[
\alpha_{\xi_2}\circ(\Ad v^*\circ\alpha_{\xi_1})=
\Ad w'\circ(\Ad v^*\circ\alpha_{\xi_1})\circ\alpha_{\xi_2}. 
\]
Then the conclusion follows from Proposition \ref{ocneanu3}. 
\end{proof}

\begin{lem}
Let $A$ be a UHF algebra of infinite type with a unique trace $\tau$ and 
let $\alpha:\Z^2\curvearrowright A$ be a strongly outer action of $\Z^2$. 
Suppose that $(u_g)_{g\in\Z^2}$ is an $\alpha$-cocycle in $A$ 
satisfying $\Delta_\tau(u_g)=0$ for any $g\in\Z^2$. 
Then there exists a continuous family $(v_t)_{t\in[1,\infty)}$ 
of unitaries in $A$ such that 
\[
\lim_{t\to\infty}v_t\alpha_g(v_t^*)=u_g
\]
holds for every $g\in\Z^2$. 
\end{lem}
\begin{proof}
By the remark following \cite[Definition 3.2]{KM}, 
$(u_g)_g$ is admissible. 
It follows from \cite[Theorem 4.7]{KM} that 
we can find a sequence of unitaries $(x_n)_n$ in $A$ such that 
$x_n\alpha_g(x_n^*)\to u_g$ for all $g\in\Z^2$. 
We may assume $\lVert x_n\alpha_{\xi_i}(x_n^*)-u_{\xi_i}\rVert<1/2$ 
for any $n\in\N$ and $i=1,2$. 
For each $n\in\N$ and $i=1,2$, 
\[
r_{i,n}=\frac{1}{2\pi\sqrt{-1}}\tau(\log(x_n^*u_{\xi_i}\alpha_{\xi_i}(x_n)))
\]
is in $\tau(K_0(A))$, because 
$\Delta_\tau(x_n^*u_{\xi_i}\alpha_{\xi_i}(x_n))=\Delta_\tau(u_{\xi_i})=0$. 
By using Lemma \ref{UHF2} twice, one can find a unitary $y_n\in A$ such that 
\[
\lVert y_n-\alpha_{\xi_i}(y_n)\rVert
<\lVert x_n\alpha_{\xi_i}(x_n^*)-u_{\xi_i}\rVert+1/n
\]
and 
\[
\tau(\log(y_n\alpha_{\xi_i}(y_n^*)))=2\pi\sqrt{-1}r_{i,n}
\]
for each $i=1,2$. 
Then 
\[
\tau(\log(y_n^*x_n^*u_{\xi_i}\alpha_{\xi_i}(x_ny_n)))
=\tau(\log(x_n^*u_{\xi_i}\alpha_{\xi_i}(x_n)))
+\tau(\log(y_n^*\alpha_{\xi_i}(y_n)))=0. 
\]
Hence, by replacing $x_n$ with $x_ny_n$, 
we may assume $\tau(\log(x_n^*u_{\xi_i}\alpha_{\xi_i}(x_n)))=0$ 
for any $n\in\N$ and $i=1,2$. 
Set $w_n=x_n^*x_{n+1}$. 
We have $\lVert w_n-\alpha_{\xi_i}(w_n)\rVert\to0$ as $n\to\infty$ 
for each $i=1,2$. 
Moreover, 
\begin{align*}
\tau(\log(w_n\alpha_{\xi_i}(w_n^*)))
&=\tau(\log(x_n^*x_{n+1}\alpha_{\xi_i}(x_{n+1}^*x_n))) \\
&=\tau(\log(x_{n+1}\alpha_{\xi_i}(x_{n+1}^*x_n)x_n^*)) \\
&=\tau(\log(u_{\xi_i}^*x_{n+1}
\alpha_{\xi_i}(x_{n+1}^*x_n)x_n^*u_{\xi_i})) \\
&=\tau(\log(u_{\xi_i}^*x_{n+1}\alpha_{\xi_i}(x_{n+1}^*)))
+\tau(\log(\alpha_{\xi_i}(x_n)x_n^*u_{\xi_i})) \\
&=0. 
\end{align*}
Therefore Lemma \ref{3torus} applies and yields 
a sequence $(\tilde w_n)_n$ of unitaries in $C([0,1],A)$ such that 
$\tilde w_n(0)=1$, $\tilde w_n(1)=w_n$ and 
\[
\lim_{n\to\infty}\sup_{t\in[0,1]}
\lVert\tilde w_n(t)-\alpha_{\xi_i}(\tilde w_n(t))\rVert=0. 
\]
Define $v:[1,\infty)\to U(A)$ 
by $v_{n+t}=x_n\tilde w_n(t)$ for $n\in\N$ and $t\in[0,1]$. 
It is easy to see that $(v_t)_{t\in[1,\infty)}$ meets the requirement. 
\end{proof}

Recall that 
the unique tracial state on $\mathcal{Z}$ is denoted by $\omega$. 

\begin{thm}\label{ocneanu5}
Let $(\alpha,u)$ be a strongly outer cocycle action of $\Z^2$ 
on the Jiang-Su algebra $\mathcal{Z}$. 
If $\Delta_\omega(u(\xi_1,\xi_2)u(\xi_2,\xi_1)^*)=0$, then 
$(\alpha,u)$ is equivalent to a genuine action. 
\end{thm}
\begin{proof}
We can perturb $u(\cdot,\cdot)$ by a $\T$-valued coboundary 
and assume $\Delta_w(u(g,h))=0$ for any $g,h\in\Z^2$ 
(see Remark \ref{OlPedTak} below). 
Let $Z\subset\mathcal{Z}$, $B_0$, $B_1$ and $B$ be 
as in the proof of Lemma \ref{prestability}. 
We denote the unique trace on $\mathcal{Z}\otimes B$ by $\tau$. 
By Theorem \ref{ocneanu4}, for each $j=0,1$, 
$(\alpha\otimes\id,u\otimes1)$ on $\mathcal{Z}\otimes B_j$ is 
equivalent to a genuine action. 
Thus there exists a family of unitaries $(w_{j,g})_{g\in\Z^2}$ 
in $\mathcal{Z}\otimes B_j$ such that 
\[
u(g,h)\otimes1=(\alpha_g\otimes\id)(w_{j,h}^*)w_{j,g}^*w_{j,gh}
\]
for all $g,h\in\Z^2$. 
In Proposition \ref{ocneanu3}, it is easy to modify the unitaries $a,b$ 
so that $\Delta_\tau(a)=\Delta_\tau(b)=0$. 
In the proof of Theorem \ref{ocneanu4}, 
one can choose the unitary $v$ so that $\Delta_\tau(v)=0$. 
Therefore we may assume $\Delta_\tau(w_{j,g})=0$ for every $g\in\Z^2$. 
Define a $\Z^2$-action $\beta:\Z^2\curvearrowright\mathcal{Z}\otimes B$ 
by $\beta_g=\Ad w_{0,g}\circ(\alpha_g\otimes\id)$. 
Clearly $(w_{1,g}w_{0,g}^*)_{g\in\Z^2}$ is 
a $\beta$-cocycle in $\mathcal{Z}\otimes B$. 
It follows from the lemma above that 
there exists a continuous family of unitaries 
$(v_t)_{t\in[0,1)}$ in $\mathcal{Z}\otimes B$ 
such that $v_t\beta_g(v_t^*)\to w_{1,g}w_{0,g}^*$ as $t\to1$ 
for all $g\in\Z^2$. 
We may further assume $v_0=1$. 
Define $\tilde w_g\in U(Z)\subset U(\mathcal{Z})$ by 
\[
\tilde w_g(t)=\begin{cases}v_t\beta_g(v_t^*)w_{0,g}&t\in[0,1)\\
w_{1,g}&t=1. \end{cases}
\]
It is straightforward to check 
\[
u(g,h)\otimes1=(\alpha_g\otimes\id)(\tilde w_h^*)\tilde w_g^*\tilde w_{gh}
\]
for all $g,h\in\Z^2$, which means that 
$(\alpha\otimes\id,u\otimes1)$ on $\mathcal{Z}\otimes Z$ 
(and hence on $\mathcal{Z}\otimes\mathcal{Z}$) is 
equivalent to a genuine action. 
By Remark \ref{absorbZca}, 
$(\alpha,u)$ on $\mathcal{Z}$ is equivalent to 
$(\alpha\otimes\id,u\otimes1)$ on $\mathcal{Z}\otimes\mathcal{Z}$, 
and so the proof is completed. 
\end{proof}

\begin{rem}\label{OlPedTak}
By \cite[Proposition 3.2]{OPT}, 
the hypothesis $\Delta_\omega(u(\xi_1,\xi_2)u(\xi_2,\xi_1)^*)=0$ 
in the theorem above is equivalent to 
triviality of the cohomology class of $[\Delta_\omega(u(\cdot,\cdot))]$ 
in $H^2(\Z^2,\T)$, 
where the range of $\Delta_\omega$ is identified with $\T$. 
\end{rem}

\bigskip

\noindent
\textbf{Acknowledgments}\\
We would like to thank the referee for a number of helpful comments.

\end{document}